\documentclass{amsart}
\allowdisplaybreaks[4]
\usepackage{color}

\newcommand{\Z}{\mathbf{Z}}
\newcommand{\R}{\mathbf{R}}
\newcommand{\C}{\mathbf{C}}
\newcommand{\sig}{\sigma}

\newcommand{{\ba}}{\bf a}
\newcommand{\ve}{\varepsilon}
\newcommand{\la}{\lambda}
\newcommand{\La}{\Lambda}
\newcommand{\ga}{\gamma}
\newcommand{\pa}{\partial}
\newcommand{\ra}{\rightarrow}
\newcommand{\Om}{\Omega}
\newcommand{\del}{\delta}
\newcommand{\Del}{\Delta}

\newcommand{\na}{\nabla}

\newcommand{\al}{\alpha}

\newcommand{\be}{\begin{equation}}
\newcommand{\ee}{\end{equation}}

\newcommand{\om}{\omega}

\newtheorem{lem}{Lemma}{\bf}{\it}
\newtheorem{remark}{Remark}{\it}{\rm}

\newtheorem{theorem}{Theorem}
\newtheorem{proposition}{Proposition}

\newtheorem{hypothesis}{Hypothesis}
\numberwithin{theorem}{section}
\numberwithin{lem}{section}
\numberwithin{hypothesis}{section}
\numberwithin{equation}{section}
\numberwithin{proposition}{section}
\numberwithin{corollary}{section}
\title[Strong Convergence II] {Strong Convergence to the homogenized limit of parabolic equations with random coefficients II}

\author{Joseph G. Conlon \  and \  Arash Fahim}

\address{ (Joseph G. Conlon): University of Michigan, Department of Mathematics, Ann Arbor,
  MI 48109-1109}
\email{conlon@umich.edu}

\address{ (Arash Fahim): University of Michigan, Department of Mathematics, Ann Arbor,
  MI 48109-1109}
\email{fahimara@umich.edu}


\begin{document}

\maketitle

\begin{abstract}
This paper is concerned with the study of solutions to discrete parabolic equations in divergence form with random coefficients, and their convergence to solutions of a homogenized equation. 
 In \cite{cf1}  rate of convergence results in homogenization and estimates on the difference between the averaged Green's function and the homogenized Green's function for random environments which satisfy a Poincar\'{e} inequality were obtained. Here these results are extended to certain environments  in which correlations can have arbitrarily small power law decay. Similar results  for discrete elliptic equations  were obtained in \cite{cf2}.

\end{abstract}

\section{Introduction.}
In this paper we continue the study of solutions to divergence form parabolic equations with random coefficients begun in \cite{cf1}.  In \cite{cf1} we were concerned with solutions  $u(x,t,\om)$ to the 
equation
\be \label{A1}
\frac{\pa u(x,t,\om)}{\pa t} \ = \ -\nabla^*{\bf a}(\tau_{x,t}\om)\nabla u(x,t,\om) \ , \quad x\in \Z^d, \ t\ge0, \ \om\in\Om,
\ee
with initial data
\be \label{B1}
u(x,0,\om) \ = \ h(x), \quad x\in\Z^d, \ \om\in\Om \ .
\ee
Here $\Z^d$ is the $d$ dimensional integer lattice and $(\Om,\mathcal{F},P)$ is a probability space equipped with  measure preserving  translation operators  $\tau_{x,t}:\Om\ra\Om, \ x\in\Z^d,t\in\R$. 
In (\ref{A1}) we take $\nabla$ to be the discrete gradient operator defined by 
\be \label{C1}
\na \phi(x)  \ = \  \big( \na_1 \phi(x),... \ \na_d\phi(x) \big), \quad  \na_i \phi(x) = \phi (x + {\bf e}_i) - \phi(x),  
\ee
 where the vector  ${\bf e}_i \in \Z^d$ has 1 as the ith coordinate and 0 for the other coordinates, $1\le i \le  d$.  Then  $\nabla$ is a $d$ dimensional { \it column} operator, with adjoint $\nabla^*$ which is a $d$ dimensional {\it row} operator.
The function  ${\bf a}:\Om\ra\R^{d(d+1)/2}$  from $\Om$ to the space of symmetric $d\times d$ matrices  satisfies the quadratic form inequality  
\begin{equation} \label{D1}
\la I_d \le {\bf a}(\om) \le \La I_d, \ \ \ \ \ \om\in\Om,
\end{equation}
where $I_d$ is the identity matrix in $d$ dimensions and $\La, \la$
are positive constants.

One expects that if the translation operators $\tau_{x,t}$ are ergodic  on $\Om$ then solutions to the random equation (\ref{A1})  converge under diffusive scaling to solutions of a constant coefficient homogenized equation. Thus suppose $f:\R^d\ra\R$ is a $C^\infty$ function with compact support and for $\ve$ satisfying $0<\ve\le 1$ set $h(x)=f(\ve x), \ x\in\Z^d$, in (\ref{B1}), and let  $u_\ve(x,t,\om)$ denote the corresponding solution to (\ref{A1})  with this initial data.  It has been shown in \cite{loy}, just assuming  ergodicity of the translation operators,  that $u_\ve(x/\ve,t/\ve^2,\om)$ converges in probability as $\ve\ra 0$ to a function $u_{\rm hom}(x,t), \ x\in\R^d, \ t>0$, which is the solution to a constant coefficient parabolic PDE
\be \label{E1}
\frac{\pa u_{\rm hom}(x,t)}{\pa t} \ = \ -\na^* {\bf a}_{\rm hom}\na u_{\rm hom}(x,t) \ , \quad x\in\R^d, \ t>0,
\ee
with initial condition
\be \label{F1}
u_{\rm hom}(x,0) \ = \ f(x), \quad x\in\R^d \ .
\ee
The $d\times d$ symmetric matrix ${\bf a}_{\rm hom}$ in (\ref{E1}) satisfies the quadratic form inequality (\ref{D1}).   Similar results under various ergodic type assumptions on $\Om$ can be found in  \cite{bmp,cps, dkl,r}. In time-independent environments the corresponding results for elliptic equations in divergence form have been proven much earlier  -see \cite{k1,k2,pv,zko}.

In \cite{cf1} we were concerned with obtaining a {\it rate of convergence} for the homogenized limit, $\lim_{\ve\ra 0} u_\ve(x/\ve,t/\ve^2,\om)=u_{\rm hom}(x,t)$. The corresponding problem for elliptic equations has been extensively studied, beginning with the seminal work of Yurinskii \cite{y}. Recent papers on the subject have addressed the issue of obtaining optimal rates of convergence \cite{go1,go2,mo}, and include results for fully non-linear elliptic equations \cite{cafs}. Optimal estimates on variances  of solutions have been obtained,  but precise  results on fluctuations analogous to the central limit theorem have been proven only in the case of one dimension \cite{bal}. In all these papers one must make a {\it quantitative strong mixing}  assumption on the environment $(\Om,\mathcal{F},P)$ in order to obtain a rate of convergence in homogenization. For the parabolic problem we were unable to find in the literature any results on rate of convergence in homogenization, except for the recent preprint \cite{mourrat} (see also \cite{gno}) in which the environment is fixed in time. In \cite{cf1} as in \cite{mourrat} our results are restricted to obtaining a rate of convergence for the mean  $\langle u_\ve(x/\ve,t/\ve^2,\cdot)\rangle$ of the  solution  of (\ref{A1}) to  $u_{\rm hom}(x,t)$.  We were able to show that, for certain environments  $(\Om,\mathcal{F},P)$ satisfying a quantitative strong mixing condition, there exists $\al>0$ depending only on $d,\La/\la$ such that
\be \label{G1}
\sup_{x\in\ve\Z^d, t>0}| \langle \ u_\ve(x/\ve,t/\ve^2,\cdot) \ \rangle - u_{\rm hom}(x,t)| \ \le \ C\ve^\al \quad {\rm for \ }  0<\ve\le 1. 
\ee

In \cite{cf1} we followed the approach of  Naddaf and Spencer \cite{ns2} to the problem of obtaining rates of convergence in homogenization by formulating the quantitative strong mixing assumption on the environment as a   Poincar\'{e} inequality.  Specifically, consider a measure space $(\tilde{\Om},\tilde{\mathcal{F}})$ of time dependent vector fields $\tilde{\om}:\Z^d\times\R\ra\R^k$ with the property that the functions $t\ra\tilde{\om}(x,t), \ t\in\R,$ are continuous for all $x\in\Z^d$ and each  $\tilde{\om}(x,t):\tilde{\Om}\ra\R^k$ is Borel measurable with respect to the $\sig$-algebra $\tilde{\mathcal{F}}$.  For a function $G:\tilde{\Om}\ra\R$ the gradient of $G$ is defined in a weak sense. Thus if $h:\Z^d\times \R\ra\R^k$ is continuous with compact support  the directional derivative $d_hG(\tilde{\om})$ of $G(\tilde{\om})$ in the direction $h$ is defined as the limit
 \be \label{H1}
 d_hG(\tilde{\om}) \ = \ \lim_{\del\ra 0} [G(\tilde{\om}+\del h)-G(\tilde{\om})]/\del \ .
 \ee
 The function $d_{\tilde{\om}}G(\tilde{\om}):\Z^d\times\R\ra\R^k$ is then the gradient of $G$ at $\tilde{\om}$ if it is Borel measurable and
   \be \label{I1}
   d_hG(\tilde{\om}) \ = \ \sum_{x\in \Z^d} \int_{-\infty}^\infty dt \ d_{\tilde{\om}}G(x,t;\tilde{\om})\cdot h(x,t) =[d_{\tilde{\om}}G(\tilde{\om}), h] \ 
  \ee
for all continuous $h:\Z^d\times \R\ra\R^k$ of compact support. In (\ref{I1}) we have denoted by $[\cdot,\cdot]$ the Euclidean inner product on $L^2(\Z^d\times\R,\R^k)$.  Letting $\|\cdot\|_2$ denote the corresponding Euclidean norm, a probability measure $\tilde{P}$ on $(\tilde{\Om},\tilde{\mathcal{F}})$ is said to satisfy a Poincar\'{e} inequality if   there is a constant $K_{\tilde{P}}>0$ such that 
 \be \label{J1}
 {\rm Var}[ G(\cdot)] \ \le \ K_{\tilde{P}} \langle \ \|d_{\tilde{\om}} G(\cdot;\tilde{\om})\|_2^2 \ \rangle  \quad 
 {\rm for \  all \  } C^1  \ {\rm functions} \  G:\tilde{\Om}\ra\C.
 \ee
 
  If the translation invariant probability measure $\tilde{P}$ is Gaussian, then the measure is determined by the $2$-point correlation function $\Gamma:\Z^d\times\R\ra\R^k\otimes\R^k$ defined by
 $\Gamma(x,t)=\langle \ \tilde{\om}(x,t)\tilde{\om}(0,0)^* \ \rangle, \ x\in\Z^d,  t\in\R$,    where $\tilde{\om}(\cdot,\cdot)\in\R^k$ is assumed to be a column vector and the superscript $*$ denotes adjoint.  Defining the Fourier transform of a function $h:\Z^d\times\R\ra\C$ by
 \be \label{K1}
 \hat{h}(\zeta,\theta) \ = \ \sum_{x\in\Z^d}\int_{-\infty}^\infty dt \  h(x,t) e^{ix\cdot\zeta+it\theta} \ , \quad \zeta\in[-\pi,\pi]^d  , \  \theta\in\R \ ,
\ee 
one can easily see that the Poincar\'{e} inequality (\ref{J1}) holds if and only if $\hat{\Gamma}\in L^\infty([-\pi,\pi]^d\times\R)$.  Hence if $\Gamma$ is integrable on $\Z^d\times\R$ then (\ref{J1}) holds, but it is unlikely to hold if $\Gamma$  is not integrable.

In the present paper we shall prove rate of convergence results in homogenization of the parabolic PDE (\ref{A1}) for certain environments  that include some Gaussian environments  in which $\Gamma$ is not integrable.  To do this we extend the method introduced in \cite{cf2} for elliptic PDE in divergence form to the parabolic case.  The idea is to consider 
environments defined by $\mathbf{a}(\om)=\tilde{\mathbf{a}}(\om(0,0))$ where $\om:\Z^d\times\R\ra\R^n$ is a {\it translation invariant function} of $\tilde{\om}:\Z^d\times\R\ra\R^k$. The gradient of $\om$ with respect to $\tilde{\om}$ is assumed to satisfy a {\it uniform integrability condition},  and the probability space  $(\tilde{\Om},\tilde{\mathcal{F}},\tilde{P})$ for $\tilde{\om}$  to satisfy the Poincar\'{e} inequality  (\ref{J1}).  

We define what we mean by the terms used in the previous paragraph. Let $(\Om,\mathcal{F},P)$ be the probability space for $\om$ induced by the probability space  $(\tilde{\Om},\tilde{\mathcal{F}},\tilde{P})$ for $\tilde{\om}$  and the functional dependence $\tilde{\om}\ra\om$.  Translation operators $\tau_{x,t}, \ x\in\Z^d,t\in\R,$ on $\Om$ are defined by $\tau_{x,t}\om(z,s)=\om(x+z,t+s), \ z\in\Z^d,s\in\R,$ with a similar definition of translation on $\tilde{\Om}$.  The function $\tilde{\om}\ra\om$, which we denote by $\om(\cdot,\cdot,\tilde{\om})$ is translation invariant if
\be \label{L1}
\tau_{x,t} \om(\cdot,\cdot,\tilde{\om}) \ = \ \om(\cdot,\cdot,\tau_{x,t}\tilde{\om}) \quad {\rm for \ } x\in\Z^d,t\in\R, \tilde{\om}\in\tilde{\Om}.
\ee
Note that if $\om$ is a {\it linear} translation invariant function of $\tilde{\om}$ then $\om$ is the convolution of some function $h:\Z^d\times\R\ra \R^n\otimes\R^k$ from $\Z^d\times\R$ to $n\times k$ matrices  with $\tilde{\om}$, 
\be \label{M1}
\om(x,t,\tilde{\om}) \ = \ h*\tilde{\om}(x,t) \ =  \ \sum_{y\in\Z^d}\int_{-\infty}^\infty ds \  h(x-y,t-s)\tilde{\om}(y,s), \quad x\in\Z^d,t\in\R. 
\ee
For given $x\in\Z^d,t\in\R$ we use the notation of (\ref{I1}) to write the gradient of the function $\om(x,t,\cdot):\tilde{\Om}\ra\R^n$ as $d_{\tilde{\om}} \om(z,s;x,t,\tilde{\om}), \ z\in\Z^d,s\in\R,\tilde{\om}\in\tilde{\Om}$. The uniform integrability condition is then that
\be \label{N1}
 \sum_{x\in\Z^d}\int_{-\infty}^\infty dt \ \left[\sup_{\tilde{\om}\in\tilde{\Om}}\left|d_{\tilde{\om}} \om(0,0;x,t,\tilde{\om})\right| \ \right]^q \ \le \ (K_{\om,q})^q <\infty
\ee
for some $q$ with $1\le q<2$.  It follows from (\ref{M1}) that when $\om$ is a linear function of $\tilde{\om}$ the condition (\ref{N1}) is equivalent to the condition that the function $h$ in (\ref{M1})  is $q$ integrable. 

In \cite{cf1} we proved that  (\ref{G1}) holds in the case where $(\Om,\mathcal{F},P)$ is the stationary process associated with a massive Euclidean field theory. This Euclidean field theory is determined by a potential $V : \R^d \ra \R$ which is a $C^2$ uniformly convex function, and a mass $m>0$.  Thus the second derivative ${\bf a}(\cdot)=V''(\cdot)$ of $V(\cdot)$ is assumed to satisfy the inequality (\ref{D1}). Consider functions 
$\phi : \Z^d\times\R \ra \R$ which we denote as $\phi(x,t)$ where $x$ lies on the integer lattice $\Z^d$ and $t$ on the real line $\R$.  Let $\Om$ be the space of
all such functions which have the property that for each $x\in\Z^d$ the function $t\ra\phi(x,t)$ on $\R$ is continuous, and $\mathcal{F}$ be the Borel algebra generated by finite 
dimensional rectangles 
$\{ \phi(\cdot,\cdot) \in \Om: \  |\phi(x_i,t_i) - a_i| < r_i, \ i=1,...,N\}$, where
$(x_i,t_i) \in \Z^d\times\R, \ a_i \in \R, \ r_i > 0, \ i=1,...,N, \ N \ge 1$.  
For any $d\ge 1$ and $m>0$ one can define \cite{c1,fs} a unique ergodic translation invariant probability 
measure $P_m$ on $(\Om, \mathcal{F})$ which depends on the function $V$ and $m$.  In this measure the variables $\phi(x,t), \ x\in\Z^d,t>0,$ conditioned on the variables $\phi(x,0), \ x\in\Z^d,$ are determined as solutions of the infinite dimensional stochastic differential  equation 
\be \label{O1}
d\phi(x,t) \ = \ -\frac{\pa}{\pa\phi(x,t)}\sum_{x'\in\Z^d} \frac{1}{2}\{V(\na\phi(x',t))+m^2\phi(x',t)^2/2\} \ dt +dB(x,t) \ , \quad x\in\Z^d, t>0, 
\ee
 where $B(x,\cdot), \ x\in\Z^d,$ are independent copies of Brownian motion. Formally the invariant measure for the Markov process (\ref{O1}) is the Euclidean field theory measure
\be \label{P1}
\exp \left[ - \sum_{x\in \Z^d} V\left( \na\phi(x)\right)+m^2\phi(x)^2/2 \right] \prod_{x\in \Z^d} d\phi(x)/{\rm normalization}.
\ee
Hence if the variables $\phi(x,0), \ x\in\Z^d,$ have distribution determined by (\ref{P1}), then $\phi(\cdot,t), \ t>0,$ is a stationary process and so can be extended to all $t\in\R$ to yield a measure $P_m$ on $(\Om,\mathcal{F})$.   The probability space $(\Om,\mathcal{F},P_m)$ satisfies the Poincar\'{e} inequality (\ref{J1}) with constant $K_{P_m}=4/m^4$.  In \cite{cf1} we conclude from this that the inequality (\ref{G1})  holds provided $\mathbf{a}(\om)=\tilde{\mathbf{a}}(\phi(0,0))$, where $\tilde{\mathbf{a}}:\R\ra\R^{d(d+1)/2}$ is assumed to be a $C^1$ function satisfying (\ref{D1}) and 
$\|D\tilde{\mathbf{a}}(\cdot)\|_\infty<\infty$. 

Let $(\tilde{\Om},\tilde{\mathcal{F}},\tilde{P})$ be a probability space for which the Poincar\'{e} inequality (\ref{J1})  holds, and $\tilde{\om}\ra\om$ a function which satisfies the translation invariant condition (\ref{L1}) and the uniform integrability condition (\ref{N1}) for some $q$ with $1\le q<2$.  Our goal in the current paper is to show that the inequality (\ref{G1}) holds for the environment $(\Om, \mathcal{F},P)$  of $\om\in\Om$
where   $\mathbf{a}(\om)=\tilde{\mathbf{a}}(\om(0,0))$ and $\tilde{\mathbf{a}}:\R\ra\R^{d(d+1)/2}$ is a $C^1$ function satisfying (\ref{D1}) and $\|D\tilde{\mathbf{a}}(\cdot)\|_\infty<\infty$.  Rather than attempt to formulate a general theorem for such environments, we shall only rigorously prove that (\ref{G1}) holds for certain limits of the probability spaces $(\Om,\mathcal{F},P_m)$ defined by (\ref{O1}), (\ref{P1}) as $m\ra 0$.  In $\S2$ we indicate the generality of our argument by showing that the proof of Proposition 6.3 of \cite{cf1} formally extends to the environment $(\Om, \mathcal{F},P)$.

From (\ref{O1}) we see that the stationary process $\om(\cdot,\cdot)=\phi(\cdot,\cdot)$ is a translation invariant function of the white noise stationary process $\tilde{\om}(\cdot,\cdot)=dB(\cdot,\cdot)$.  It is well known that the white noise process satisfies a Poincar\'{e} inequality (\ref{J1}) with $K_{\tilde{P}}=1$. Consider now the terminal value problem for the backwards in time parabolic PDE
\begin{eqnarray} \label{Q1}
\frac{\pa u(z,s)}{\pa s}  \ &=&  \ \frac{1}{2}\na^*V''(\na\phi(z,s))\na u(z,s), \quad   s<t,z\in\Z^d, \\
 u(z,t) \ &=& \ u_0(z),  \quad  z\in\Z^d, \nonumber
\end{eqnarray}
with solution
\be \label{R1}
u(z,s) \ = \ \sum_{x\in \Z^d} G(z,s;x,t,\phi) u_0(x) \ , \quad s<t,z\in\Z^d.
\ee
 We see from  \cite{cf1} that the gradient of $\om(x,t)=\phi(x,t), \ x\in\Z^d,t\in\R,$ with respect to $\tilde{\om}$ should be  given by the formula
\begin{eqnarray} \label{T1}
d_{\tilde{\om}}\om(z,s;x,t,\tilde{\om}) \ &=& \  e^{-m^2(t-s)/2}G(z,s;x,t,\phi) \quad {\rm for \ } s<t, \\
d_{\tilde{\om}}\om(z,s;x,t,\tilde{\om}) \ &=& \  0 \quad {\rm for \ } s>t, \quad z\in\Z^d. \nonumber
\end{eqnarray}
In \cite{gos} a discrete version of the  Aronson  inequality \cite{aronson}  was  proven in the case when  the diffusion matrix $V''(\cdot)$ for (\ref{Q1}) is {\it diagonal}. In particular it was shown that there is a positive constant $C$ depending only on $d,\La/\la$ such that
\be \label{U1}
0< \ G(z,s;x,t,\phi) \ \le \ \frac{C}{[\La(t-s)+1]^{d/2}} \exp\left[-\frac{ |x-z|}{\sqrt{\La (t-s)+1}} \ \right] \ .
\ee
Hence (\ref{T1}), (\ref{U1}) imply that the uniform integrability condition (\ref{N1}) holds  for any $q$ with $q>1+2/d$ and  the bound on the RHS of (\ref{N1}) can be taken independent of $m$ as $m\ra 0$.  Hence if $d\ge 3$ the condition (\ref{N1}) holds  in the limit $m\ra 0$ for some $q$ with $1\le q<2$. 

It has been shown by Funaki and Spohn \cite{fs} (see also \cite{c1}) that if $d\ge 3$ then there is a unique limit as $m\ra 0$ of the stationary process defined by (\ref{O1}), (\ref{P1}). In $\S3$ we shall extend the rate of convergence results in homogenization of (\ref{A1}), (\ref{B1}) obtained in \cite{cf1} for the massive field stationary process (\ref{O1}), (\ref{P1})  with $m>0$   to this massless $m\ra 0$ stationary process. In particular we prove the following:
\begin{theorem}
 Let $V:\R^d\ra\R$ be a $C^2$ function such that $V''(z), \ z\in\R^d,$ is a diagonal $d\times d$ matrix which   satisfies the quadratic form inequality  (\ref{D1}) with $\mathbf{a}(\cdot)=V''(\cdot)$.  Let $\tilde{{\bf a}}:\R\ra\R^{d(d+1)/2}$  be a $C^1$ function   on $\R$ with values in the space of symmetric $d\times d$ matrices which satisfies the quadratic form inequality  (\ref{D1}) and has bounded first derivative $D\tilde{{\bf a}}(\cdot)$ so  $\|D\tilde{{\bf a}}(\cdot)\|_{\infty}<\infty$.  For $d\ge 3$ let $(\Om, \mathcal{F}, P)$ be the probability space of the stationary process $\phi(\cdot,\cdot)$ determined by the limit as $m\ra 0$ of the stationary process defined by (\ref{O1}), (\ref{P1}), and set ${\bf a}(\cdot)$ in (\ref{A1}) to be ${\bf a}(\phi)=\tilde{{\bf a}}(\phi(0,0)), \ \phi\in\Om$. 
 Let  $f:\R^d\ra\R$ be a $C^\infty$ function of compact support,  $u_\ve(x,t,\om)$  the corresponding solution to (\ref{A1}), (\ref{B1}) with $h(x)=f(\ve x), \ x\in\Z^d,$ and $u_{\rm hom}(x,t), \ x\in\R^d,t>0,$ the solution to (\ref{E1}), (\ref{F1}).   Then  there is a constant $\al>0$ depending only on $d,\La/\la$ and a constant $C$ depending only on $d,\La,\la, \|D\tilde{{\bf a}}(\cdot)\|_{\infty}, f(\cdot)$ such that (\ref{G1}) holds. 
 \end{theorem}
 \begin{remark}
 The exponent $\al>0$ in (\ref{G1}) can be taken equal to $1$ if $d\ge 5$  and the ratio $\la/\La$ is sufficiently close to $1$.  In \cite{cf1} the matrix $V''(\cdot)$ is not required to be diagonal since we use the fact that the Poincar\'{e} inequality (\ref{J1}) holds for the massive field stationary process.  In the  Gaussian case where $V(\cdot)$ is quadratic (\ref{G1}) also holds without the restriction that $V''(\cdot)$ be diagonal.  This follows from the fact that a bound  on the Green's function defined by (\ref{R1}) similar to (\ref{U1}) holds in this case. Another way of seeing it is  to note that the field $\phi(\cdot,\cdot)$ is  a linear translation invariant function of another field $\tilde{\om}(\cdot,\cdot)$ as in (\ref{M1}) with probability space which does satisfy a Poincar\'{e} inequality. This property of $\phi(\cdot,\cdot)$, being the convolution of a function with another field whose probability space  satisfies a Poincar\'{e} inequality, does not seem to generalize to the case of uniformly convex $V(\cdot)$  which is not quadratic.  One reason for this is that the measure for the stationary process $\phi(\cdot,\cdot)$ associated with (\ref{O1}), (\ref{P1}) appears to be log concave in $\phi(\cdot,\cdot)$ only in the Gaussian case when $V(\cdot)$ is quadratic (see Appendix).  In contrast,   the  invariant measure (\ref{P1}) for the stationary process is easily seen to be log concave  when $V(\cdot)$ is convex. Hence the Brascamp-Lieb inequality \cite{bl} implies that  a Poincar\'{e} inequality holds for the gradient $\na\phi(\cdot)$  of the invariant measure field $\phi(\cdot)$ of (\ref{P1}) if $V(\cdot)$ is uniformly convex (see \cite{cs}) . 
 \end{remark}
 Parallel to \cite{cf1} we also establish for the massless field stationary process  {\it point-wise} convergence at large length scales of the averaged Green's function for the initial value problem (\ref{A1}), (\ref{B1}) to the homogenized Green's function for the initial value problem (\ref{E1}), (\ref{F1}).  The averaged Green's function $G_{\mathbf{a}}(x,t), \ x\in\Z^d,t\ge 0,$ for (\ref{A1}), (\ref{B1}) is defined by $G_{\mathbf{a}}(x,t)=\langle \ u(x,t,\cdot) \ \rangle$, where $h(\cdot)$ in (\ref{B1}) is the Kronecker delta function $h(x)=0$ if $x\ne 0$ and $h(0)=1$. 
\begin{theorem}  With the same environment as in the statement of Theorem 1.1,  let $G_{{\bf a}_{\rm hom}}(x,t), \ x\in\R^d, t>0,$ be the Green's function for the homogenized problem (\ref{E1}), (\ref{F1}). Then there are constants $\al,\ga>0$ depending only on $d$ and the ratio $\La/\la$ of the constants $\la,\La$ of (\ref{D1}), and a constant $C$ depending only on $\|D\tilde{{\bf a}}(\cdot)\|_{\infty},\La,\la,d$ such that  for $\La t\ge 1$,
\be \label{V1}
\big| G_{{\bf a}}(x,t)-G_{{\bf a}_{\rm hom}}(x,t)\big| \ \le \ \frac{C}{[\La t+1]^{(d+\alpha)/2}} \exp\left[-\ga\min\left\{|x|, \ \frac{ |x|^2}{\La t+1}\right\}\right] \ ,
\ee
\be \label{W1}
\big| \na G_{{\bf a}}(x,t)-\na G_{{\bf a}_{\rm hom}}(x,t)\big| \ \le \ \frac{C}{[\La t+1]^{(d+1+\alpha)/2}} \exp\left[-\ga\min\left\{|x|, \ \frac{ |x|^2}{\La t+1}\right\}\right] \ , 
\ee
\be \label{X1}
\big|\na\na  G_{{\bf a}}(x,t)-\na\na G_{{\bf a}_{\rm hom}}(x,t)\big| \ \le \ \frac{C}{[\La t+1]^{(d+2+\alpha)/2}} \exp\left[-\ga\min\left\{|x|, \ \frac{ |x|^2}{\La t+1}\right\}\right] \ .  
\ee
\end{theorem}
 The limit as $m\ra 0$ of the invariant measure (\ref{P1}) is a probability measure on gradient fields $\om:\Z^d\ra\R^d$, where formally $\om(x)=\na\phi(x), \ x\in\Z^d$. This massless field theory measure is ergodic with respect to translation operators \cite{c1,fs} for all $d\ge 1$. In the case $d=1$ it has a simple structure since  then the variables $\om(x), \ x\in\Z$, are i.i.d.  Note that in the probability space $(\Om,\mathcal{F},P)$ for the massless field theory,  the Borel algebra $\mathcal{F}$ is  generated by the intersection of finite dimensional rectangles and the hyperplanes imposing the gradient constraints for $\om(\cdot)$.  For $d\ge 3$ the gradient field theory measure induces a measure on fields $\phi:\Z^d\ra\R$ which is simply the limit of the measures (\ref{P1})  as $m\ra 0$.   For $d=1,2$ the $m\ra 0$ limit of the measures (\ref{P1}) on fields  $\phi:\Z^d\ra\R$ does not exist. Naddaf and Spencer showed in \cite{ns1} that the 2-point correlation function for the massless field can be represented via the Helffer-Sj\"{o}strand formula \cite{hs}  as the expectation value of a Green's function for a divergence form PDE with random coefficients. Using this fact and techniques of homogenization theory they were able to prove that averages of the function $x\ra   \langle \ \phi(x)\phi(0) \ \rangle$ over large length scales converge to the solution of a constant coefficient elliptic PDE
 \be \label{Y1}
\na^* {\bf a}_{\rm hom}\na u_{\rm hom}(x) \ = \  f(x), \quad x\in\R^d \ .
\ee
 Using the techniques of the current paper we obtain a convergence theorem for the correlation function $\langle \ \phi(x)\phi(0) \ \rangle$  which is {\it pointwise} in $x$: 
 \begin{theorem}
  Let $V:\R^d\ra\R$ be a $C^3$ function such that $V''(z), \ z\in\R^d,$ is a diagonal $d\times d$ matrix which   satisfies the quadratic form inequality  (\ref{D1}) with $\mathbf{a}(\cdot)=V''(\cdot)$ and also $\|V'''(\cdot)\|_\infty<\infty$.    Let $G_{{\bf a}_{\rm hom}}(x), \ x\in\R^d, $ be the Green's function for the Naddaf-Spencer PDE (\ref{Y1}), and $\langle\cdot\rangle$ denote the massless field theory measure for the $m\ra 0$ limit of (\ref{P1}). Then for $d\ge 2$ there is a constant $\al>0$ depending only on $d$ and the ratio $\La/\la$ of the constants $\la,\La$ of (\ref{B1}), and a constant $C$ depending only on $\|V'''(\cdot)\|_{\infty},\La,\la,d$ such that  for $x\in\Z^d-\{0\}$,
  \be \label{Z1}
 \big| \ \langle \ \phi(x)\phi(0) \ \rangle - G_{{\bf a}_{\rm hom}}(x) \ \big| \ \le  \ C/|x|^{d-2+\al} \ ,
  \ee
   \be \label{AA1}
 \big| \ \langle \ \na\phi(x)\phi(0) \ \rangle - \na G_{{\bf a}_{\rm hom}}(x) \ \big| \ \le  \ C/|x|^{d-1+\al} \ ,
  \ee
\be \label{AB1}
 \big| \ \langle \ \na\na^*\phi(x)\phi(0) \ \rangle - \na\na^* G_{{\bf a}_{\rm hom}}(x) \ \big| \ \le  \ C/|x|^{d+\al} \ .
  \ee 
   \end{theorem} 
\begin{remark}
Note that both terms $\langle \ \phi(x)\phi(0) \ \rangle$ and $G_{{\bf a}_{\rm hom}}(x) $ on the LHS of (\ref{Z1}) are  divergent in dimension $d=2$. However the difference suitably defined is finite. The exponent $\al>0$ can be taken arbitrarily close to $1$ by choosing $\la/\La$ sufficiently close to $1$. 
\end{remark}
In the case when
$V(z) \ = \ |z|^2/2+ a \sum_{j=1}^d\cos z_j, \quad z\in\R^d$, the $m\ra 0$ probability measure (\ref{P1}) describes the dual representation of a gas of lattice dipoles with activity $a$ (see  \cite{brydges}).
The inequality (\ref{Z1}) was proven by Dimock and Hurd  (Theorem 2 of \cite{dh}) for a continuous space version of the dipole gas under the assumption that the activity $a$ is sufficiently small. It is not possible to obtain  from the renormalization group method they use a reasonable estimate on the value of $a$ for which their theorem holds. 
\vspace{.1in}

\section{Variance Estimate on the Solution to a PDE on $\Om$}
We recall some definitions from \cite{cf1}. For $\xi\in\R^d$  and $1\le j\le d$ we define the $\xi$ derivative of a measurable function  $\psi:\Om\ra\C$  in the $j$ direction by $\pa_{j,\xi}$, and its adjoint by  $\pa^*_{j,\xi}$, where
\begin{eqnarray} \label{A2}
\pa_{j,\xi} \psi(\om) \ &=& \  e^{-i{\bf e}_j.\xi}\psi(\tau_{{\bf e}_j,0}   \ \om)-\psi(\om),  \\
\pa^*_{j,\xi} \psi(\om) \ &=& \   e^{i{\bf e}_j.\xi}\psi(\tau_{-{\bf e}_j,0} \ \om)-\psi(\om). \nonumber
\end{eqnarray}
We also define a $d$ dimensional column $\xi$ gradient operator  $\pa_\xi$ by  $\pa_\xi=(\pa_{1,\xi},....,\pa_{d,\xi})$, which has adjoint  $\pa_\xi^*$  given by the row operator $\pa_\xi^*=(\pa^*_{1,\xi},....,\pa^*_{d,\xi})$. The time derivative of $\psi$ is defined by
\be \label{B2}
\pa \psi(\om) \ = \ \lim_{\del\ra 0}[\psi(\tau_{0,\del}\om)-\psi(\om)]/\del \ .
\ee 
Let $\mathcal{H}(\Om)$ be the Hilbert space of measurable functions $\Psi:\Om\ra\C^d$ with norm $\|\Psi\|_{\mathcal{H}(\Om)}$ given by $\|\Psi\|_{\mathcal{H}(\Om)}^2=\langle \ |\Psi(\cdot)|_2^2 \ \rangle$, where $|\cdot|_2$ is the Euclidean norm on $\C^d$. Then there is a unique row vector solution $\Phi(\xi,\eta,\om)=\big(\Phi_1(\xi,\eta,\om),...,\Phi_d(\xi,\eta,\om)\big)$ to the equation
\be \label{C2}
[\eta+\pa]\Phi(\xi,\eta,\om)+\pa_\xi^*{\bf a}(\om)\pa_\xi\Phi(\xi,\eta,\om)=-\pa^*_\xi {\bf a}(\om), \quad \eta>0, \ \xi\in \R^d, \ \om\in\Om,
\ee
such that $\Phi(\xi,\eta,\cdot)v\in L^2(\Om)$ for any $v\in\C^d$. Furthermore $\Phi(\xi,\eta,\cdot)v\in L^2(\Om)$  satisfies the inequality
\be \label{D2}
\eta\|\Phi(\xi,\eta,\cdot)v\|^2_{L^2(\Om)}+\la\|\pa_\xi\Phi(\xi,\eta,\cdot)v\|^2_{\mathcal{H}(\Om)} \ \le \ 
\La^2|v|^2/\la \ .
\ee
Letting $\mathcal{P}$ denote the projection orthogonal to the constant function, our generalization of Proposition 6.3 of \cite{cf1}  is as follows: 
\begin{proposition}
Suppose ${\bf a}(\cdot)$ in (\ref{C2}) is given by   $\mathbf{a}(\om)=\tilde{\mathbf{a}}(\om(0,0))$ where $\tilde{\mathbf{a}}:\R^n\ra\R^{d(d+1)/2}$ is a $C^1$ $d\times d$ symmetric matrix valued function  satisfying the quadratic form inequality (\ref{D1}) and $\|D\tilde{\mathbf{a}}(\cdot)\|_\infty<\infty$.  The random field $\om:\Z^d\times\R\ra\R^n$ is a translation invariant function of a random field $\tilde{\om}:\Z^d\times\R\ra\R^k$ which satisfies the uniform integrability condition (\ref{N1}) for some $q$ with $1\le q<2$.   The probability space $(\tilde{\Om},\tilde{\mathcal{F}},\tilde{P})$ of fields $\tilde{\om}:\Z^d\times\R\ra\R^k$  is assumed to satisfy the Poincar\'{e} inequality (\ref{J1}). Then there exists $q_0<2$ depending only on $d,\La/\la$ such that if $q_0\le q\le 2$ and  $g\in L^p(\Z^d\times\R,\C^d\otimes\C^d)$ with $p=2q/(3q-2)$  the inequality
\be \label{E2}
\|\mathcal{P}\sum_{x\in\Z^d} \int_{-\infty}^\infty dt \ g(x,t)\pa_\xi\Phi(\xi,\eta,\tau_{x,-t}\cdot)v\|_{\mathcal{H}(\Om)}   \ \le \ \frac{CK^{1/2}_{\tilde{P}}\|D\tilde{\mathbf{a}}(\cdot)\|_\infty|v|}{\La}K_{\om,q} \|g\|_p \ , \quad v\in\C^d,
\ee
holds for a constant $C$ depending only on $d,n,k,\La/\la,q_0$. 
\end{proposition}
\begin{proof}
From (\ref{J1}) we have that
\begin{multline} \label{F2}
\|\mathcal{P}\sum_{x\in\Z^d} \int_{-\infty}^\infty dt \ g(x,t)\pa_\xi\Phi(\xi,\eta,\tau_{x,-t}\cdot)v\|_{\mathcal{H}(\Om)}  ^2   \\ \le \ 
K_{\tilde{P}}\sum_{z\in \Z^d}\int_{-\infty}^\infty ds \  \left\langle\left| \ \frac{\pa}{\pa \tilde{\om}(z,-s)}  \  \sum_{x\in\Z^d} \int_{-\infty}^\infty dt \ g(x,t)\pa_\xi\Phi(\xi,\eta,\tau_{x,-t}\cdot)v\right|_2^2\right\rangle  \ . 
\end{multline}
From the chain rule we see that
\be \label{G2}
\frac{\pa}{\pa \tilde{\om}(z,-s)} \pa_\xi\Phi(\xi,\eta,\tau_{x,-t}\cdot)v \ = \  \sum_{y\in\Z^d} 
  \int_{-\infty}^\infty dt'\ \left[\frac{\pa}{\pa \om(y,t')}\pa_\xi\Phi(\xi,\eta,\tau_{x,-t}\cdot)v\right] d_{\tilde{\om}}\om(z,-s; y,t',\tilde{\om})\ .
\ee

Next we do a translation of the functions on the RHS of (\ref{G2}). Translation of a function $G:\tilde{\Om}\ra\C$ through $(x,t)\in\Z^d\times\R$ is defined by $\tau_{x,t}G(\tilde{\om})=G(\tau_{x,t}\tilde{\om}), \ \tilde{\om}\in\tilde{\Om}$. For a function $G:\Om\ra\C$ there are two possible notions of translation through $(x,t)$, the first being given by $\tau_{x,t}G(\om)=G(\tau_{x,t}\om), \ \om\in\Om$. Since $\om$ is  a function of $\tilde{\om}\in\tilde{\Om}$ we can also define translation through $(x,t)$ by regarding $G:\Om\ra\C$ as a function on $\tilde{\Om}$ and doing the translation on $\tilde{\Om}$. It follows from the translation invariance property (\ref{L1}) that both of these notions are the same. 
Now using the translation invariance of the probability measure $\tilde{P}$ on $\tilde{\Om}$ we conclude from (\ref{F2}), (\ref{G2}) that
\begin{multline} \label{H2}
\|\mathcal{P}\sum_{x\in\Z^d} \int_{-\infty}^\infty dt \ g(x,t)\pa_\xi\Phi(\xi,\eta,\tau_{x,-t}\cdot)v\|_{\mathcal{H}(\Om)}  ^2   \ \le \
K_{\tilde{P}}\sum_{z\in \Z^d} \int_{-\infty}^\infty ds \\  \left\langle\left|   \ \sum_{x\in\Z^d} \int_{-\infty}^\infty dt  \  g(x,t)\sum_{y\in\Z^d} 
  \int_{-\infty}^\infty dt'\ \left[\tau_{-z,s}\frac{\pa}{\pa \om(y,t')}\pa_\xi\Phi(\xi,\eta,\tau_{x,-t}\cdot)v\right] \tau_{-z,s}d_{\tilde{\om}}\om(z,-s; y,t',\tilde{\om})  \ \right|_2^2  \right\rangle  \ . 
\end{multline}

We define a function $u:\Z^d\times\R\times\tilde{\Om}\ra\C^k$ by 
\be \label{I2}
u(z,s,\tilde{\om}) \ = \ e^{-iz\cdot\xi}\sum_{y\in\Z^d}  \int_{-\infty}^\infty dt' \  \left[
 d_\om\Phi(y,t';\xi,\eta,\tau_{z,-s}\om)v\right]d_{\tilde{\om}}\om(0,0; y+z,t'-s,\tilde{\om}) \ ,
\ee
where $ d_\om\Phi(\cdot;\xi,\eta,\om)v:\Z^d\times\R\ra\C^n$ is the gradient of $\Phi(\xi,\eta,\om)v$ with respect to $\om\in\Om$.
Observe now from (\ref{L1})  that
\begin{multline} \label{J2}
d_{\tilde{\om}}\om(z,-s; y,t',\tilde{\om})   \ = \ \frac{\pa}{\pa \om(z,-s)}\om(y,t',\tilde{\om}) \ = \\ 
\frac{\pa}{\pa \om(z,-s)}\om(y-z,t'+s,\tau_{z,-s} \ \tilde{\om}) \ = \ 
d_{\tilde{\om}}\om(0,0; y-z,t'+s,\tau_{z,-s} \ \tilde{\om}) \ .
\end{multline}
Hence we have that
\begin{multline} \label{K2}
e^{i(z-x).\xi}\sum_{y\in\Z^d} 
  \int_{-\infty}^\infty dt'\ \left[\tau_{-z,s}\frac{\pa}{\pa \om(y,t')}\Phi(\xi,\eta,\tau_{x,-t} \ \om)v\right] \tau_{-z,s}d_{\tilde{\om}}\om(z,-s; y,t',\tilde{\om}) \ = \\
 e^{i(z-x).\xi}\sum_{y\in\Z^d}  \int_{-\infty}^\infty dt' \ \left[
 d_\om\Phi(y-x,t'+t;\xi,\eta,\tau_{x-z,s-t} \ \om)v\right]d_{\tilde{\om}}\om(0,0; y-z,t'+s,\tilde{\om}) \ = \\
 e^{i(z-x).\xi}\sum_{y\in\Z^d}  \int_{-\infty}^\infty dt' \ \left[
 d_\om\Phi(y,t';\xi,\eta,\tau_{x-z,s-t} \ \om)v\right]d_{\tilde{\om}}\om(0,0; y+x-z,t'+s-t,\tilde{\om})  \\
 = \ u(x-z,t-s,\tilde{\om}) \ .
\end{multline}
It follows from (\ref{K2}) that (\ref{H2}) can be rewritten as
\begin{multline} \label{L2}
\|\mathcal{P}\sum_{x\in\Z^d} \int_{-\infty}^\infty dt \ g(x,t)\pa_\xi\Phi(\xi,\eta,\tau_{x,-t}\cdot)v\|_{\mathcal{H}(\Om)}  ^2   \ \le \\
K_{\tilde{P}}\sum_{z\in \Z^d}  \int_{-\infty}^\infty ds \ \langle \ |   \ \sum_{x\in\Z^d} \int_{-\infty}^\infty dt \  g(x,t)e^{i(x-z)\cdot\xi}\na u(x-z,t-s,\cdot) \ |_2^2 \ \rangle \ .
\end{multline}

In \cite{cf1} we  defined the $\xi$ derivative of a measurable function  $\psi:\Z^d\times\R\times\Om\ra\C$  in the $j$ direction by $D_{j,\xi}$, and its adjoint by  $D^*_{j,\xi}$, where
\begin{eqnarray} \label{M2}
D_{j,\xi} \psi(x,t;\om) \ &=& \  e^{-i{\bf e}_j.\xi}\psi(x-\mathbf{e}_j,t;\tau_{{\bf e}_j}\om)-\psi(x,t;\om),  \\
D^*_{j,\xi} \psi(x,\om) \ &=& \   e^{i{\bf e}_j.\xi}\psi(x+\mathbf{e}_j,t;\tau_{-{\bf e}_j}\om)-\psi(x,t;\om). \nonumber
\end{eqnarray}
The corresponding $d$ dimensional column $\xi$ gradient operator  $D_\xi$ is then given by  $D_\xi=(D_{1,\xi},....,D_{d,\xi})$, and it has adjoint  $D_\xi^*$  given by the row operator $D_\xi^*=(D^*_{1,\xi},....,D^*_{d,\xi})$.  We also defined the time derivative $D_0$ of $\psi:\Z^d\times\R\times\Om\ra\C$ by
\be \label{N2}
D_0\psi(x,t;\om) \ = \ \lim_{\del\ra 0}[\psi(x,t-\del;\tau_{0,\del}\om)-\psi(x,t;\om)]/\del \ .
\ee
We see from (\ref{A2}), (\ref{B2})  that these operators satisfy the identities
\begin{eqnarray} \label{O2}
\frac{\pa}{\pa\om(y,t)} \pa_\xi\psi(\om) \ &=& \ D_\xi d_\om\psi(y,t;\om), \quad y\in\Z^d,t\in\R,\om\in\Om, \\
\frac{\pa}{\pa\om(y,t)} \pa\psi(\om) \ &=& \ D_0 d_\om\psi(y,t;\om), \quad y\in\Z^d,t\in\R,\om\in\Om, \nonumber
\end{eqnarray}
 for  differentiable  functions $\psi:\Om\ra\C$. A similar relationship holds  for the adjoints $\pa^*_\xi,D^*_\xi$. Hence on taking the gradient of  equation (\ref{C2}) with respect to $\om(\cdot)$ we conclude from (\ref{O2}) that 
\begin{multline} \label{P2}
[\eta +D_0]\   d_\om\Phi(y,t';\xi,\eta,\om)v+D_\xi^*\tilde{{\bf a}}(\om(0,0))D_\xi  \  d_\om\Phi(y,t';\xi,\eta,\om)v \\
 = \ -D_\xi^*[ \ \del(y,t') D\tilde{{\bf a}}(\om(0,0))\{v+\pa_\xi \Phi(\xi,\eta,\om)v\} ] \quad {\rm for \ } y\in\Z^d,t'\in\R,\om\in\Om.
\end{multline}

Evidently (\ref{P2}) holds with $\om\in\Om$ replaced by $\tau_{z,-s}\om$ for any $z\in\Z^d,s\in\R$. We now multiply (\ref{P2}) with $\tau_{z,-s}\om$ in place of $\om$ on the right by $e^{-iz\cdot\xi}d_{\tilde{\om}}\om(0,0;y+z,t'-s,\tilde{\om})$, sum with respect to $y\in\Z^d$ and integrate with respect to $t'\in\R$. It then follows from (\ref{I2}), (\ref{P2}) that 
\be \label{Q2}
\eta \ u(z,s,\tilde{\om})-\frac{\pa u(z,s,\tilde{\om})}{\pa s}+\na^*\tilde{{\bf a}}(\om(z,-s))\na u(z,s,\tilde{\om}) \ = \ -\na^*f(z,s,\tilde{\om}) \ ,
\ee
where the function $f:\Z^d\times\R\times\tilde{\Om}\ra\C^d\otimes\C^k$ is given by the formula
\be \label{R2}
f(z,s,\tilde{\om}) \ = \  D\tilde{{\bf a}}(\om(z,-s))\{v+\pa_\xi \Phi(\xi,\eta,\tau_{z,-s}\om)v\} e^{-iz\cdot\xi} d_{\tilde{\om}}\om(0,0;z,-s,\tilde{\om}) \ .
\ee
For any $1<q<\infty$ we  consider the function $f$ as a mapping $f:\Z^d\times\R\ra  L^2(\tilde{\Om},\C^d\otimes\C^k)$  with norm defined by
\be \label{S2}
\|f\|_q^q \ = \  \sum_{y\in\Z^d}\int_{-\infty}^\infty dt' \|f(y,t',\cdot)\|_2^q \ ,
\ee
where $\|f(y,t',\cdot)\|_2$ is the norm of $f(y,t'\cdot)\in L^2(\tilde{\Om},\C^d\otimes\C^k)$. 
Now from (\ref{D2}) it follows that $\pa_\xi \Phi(\xi,\eta,\cdot)v\in \mathcal{H}(\Om)$ and  $\|\pa_\xi \Phi(\xi,\eta,\cdot)v\|_{\mathcal{H}(\Om)}\le \La|v|/\la$. Hence if the inequality (\ref{N1}) holds  then the function $f$ is in $L^q(\Z^d\times\R,  L^2(\tilde{\Om},\C^d\otimes\C^k))$ and $\|f\|_q\le \|D\tilde{{\bf a}}(\cdot)\|_\infty(1+\La/\la)|v|K_{\om,q}$.  We see  from (\ref{Q2}) that in the case $q=2$ then  $\na u$ is also in  $L^2(\Z^d\times\R,  L^2(\tilde{\Om},\C^d\otimes\C^k))$ and $\|\na u\|_2\le \|f\|_2/\la$. It follows now from (\ref{L2}) that  (\ref{E2}) holds with $q=2$ and $p=1$.  

To prove the inequality for some $p>1$ we use the parabolic version of Meyer's theorem \cite{m}. We note that just as the Calderon-Zygmund theorem applies to functions with range in a Hilbert space \cite{stein},  Jone's theorem for parabolic multipliers \cite{j} also applies to  functions with range in a Hilbert space (see \cite{cf1}).  We conclude  that there exists $q_0$ depending only on $d,\La/\la$ with $1<q_0<2$ such that if  $\|f\|_{q}<\infty$  for any $q$ satisfying $q_0\le q\le 2$ then $\|\na u\|_q\le2\|f\|_q/\la$. Assume now that (\ref{N1}) holds for some $q$ in the interval  $q_0\le q\le 2$. Then by Young's inequality for convolutions we see  from (\ref{L2}) that (\ref{E2}) holds with $p=2q/(3q-2)$.   
\end{proof}

\vspace{.1in}

\section{Proof of Theorem 1.1 and Theorem 1.2}
The basic approach of \cite{cf1} is to use the fact that the solution to (\ref{A1}) can be expressed by a Fourier inversion formula. For $\eta\in\C$ denote its real part by $\Re\eta\in\R$ and its imaginary part by $\Im\eta\in\R$ so that $\eta=\Re\eta+i\Im\eta$, and similarly denote the real and imaginary parts of $\xi\in\C^d$ by $\Re\xi,\Im\xi\in\R^d$  whence $\xi=\Re\xi+i\Im\xi$.    We consider solutions to  the equation 
\be \label{A3}
[\eta+\pa]\Phi(\xi,\eta,\om)+\mathcal{P}\pa_\xi^*{\bf a}(\om)\pa_\xi\Phi(\xi,\eta,\om) \ = \ -\mathcal{P}\pa^*_\xi {\bf a}(\om), \quad  \Re\eta>0, \ \xi\in \R^d, \ \om\in\Om.
\ee
As with (\ref{C2}) there exists a unique solution to (\ref{A3}) such that $\Phi(\xi,\eta,\cdot)v\in L^2(\Om)$ for any $v\in\C^d$. Furthermore $\Phi(\xi,\eta,\cdot)v\in L^2(\Om)$  satisfies the inequality
\be \label{B3}
\Re\eta\|\Phi(\xi,\eta,\cdot)v\|^2_{L^2(\Om)}+\la\|\pa_\xi\Phi(\xi,\eta,\cdot)v\|^2_{\mathcal{H}(\Om)} \ \le \  \La^2|v|^2/\la \ .
\ee
 If $\xi=0$ the  solution $\Phi(\xi,\eta,\om)$ to (\ref{C2}) has zero mean so $\langle \ \Phi(0,\eta,\cdot) \ \rangle=0$.  Hence the solutions to (\ref{C2}), (\ref{A3}) coincide if $\xi=0$  but are in general different.  For $\xi\in\R^d$ and $\eta\in\C$ with $\Re\eta>0$ let $e(\xi)\in\C^d$ be the vector $e(\xi)=\pa_\xi 1$ and $q(\xi,\eta)$ be the $d\times d$ matrix
\be \label{C3}
q(\xi,\eta) \ = \  \langle \ {\bf a}(\cdot) \ \rangle+ \langle \ {\bf a}(\cdot)\pa_\xi  \Phi(\xi,\eta,\cdot) \ \rangle \ ,
\ee
where  $\Phi(\xi,\eta,\om)$  is the solution to (\ref{A3}). The solution to (\ref{A1}), (\ref{B1}) is  shown  in \cite{c2} to be given by the formula 
\be \label{D3}
u(x,t,\om)  =  \frac{1}{(2\pi)^{d+1}}\int_{[-\pi,\pi]^d} \int_{-\infty}^\infty
\frac{\hat{h}(\xi)e^{-i\xi.x+\eta t}}{\eta+e(\xi)^*q(\xi,\eta)e(\xi)}\left[1+\Phi(\xi,\eta,\tau_{x,t}\om)e(\xi)\right] \ d[\Im\eta] \ d\xi  \ .   
\ee
If the environment $(\Om,\mathcal{F},P)$ is ergodic then the limit  $\lim_{\eta\ra 0} q(0,\eta)=\mathbf{a}_{\rm hom}$ exists, and  $\mathbf{a}_{\rm hom}$ is the diffusion matrix for the homogenized equation (\ref{E1}). Let  $\hat{G}_{\mathbf{a}}(\xi,\eta), \ \xi\in[-\pi,\pi]^d, \ \Re\eta>0,$  be the Fourier-Laplace transform of  the averaged Green's function $G_{{\bf a}}(x,t), \ x\in\Z^d,t\ge 0,$ for (\ref{A1}), (\ref{B1}) defined by
\be \label{E3}
\hat{G}_{{\bf a}}(\xi,\eta) \ = \  \int_{0}^\infty dt\sum_{x\in\Z^d} G_{{\bf a}}(x,t)\exp[ix.\xi-\eta t] \ .
\ee
 It follows from (\ref{D3}) that $\hat{G}_{\mathbf{a}}(\xi,\eta)$ is given by the formula
\be \label{F3}
\hat{G}_{\mathbf{a}}(\xi,\eta) \ = \  1/[\eta+e(\xi)^*q(\xi,\eta)e(\xi)] \quad {\rm for \ } \xi\in[-\pi,\pi]^d \ , \Re\eta>0.
\ee
In \cite{cf1} it was shown (see especially $\S3$ and Theorem 4.2 of \cite{cf1}) that Theorem 1.1 and Theorem 1.2 are  consequences of the following:
\begin{hypothesis} 
For $\xi\in\C^d, \ \eta\in\C$  there exist positive constants $C_1$ and  $\al\le1$ depending only on $d$ and $\La/\la$,  such the function  $q(\xi,\eta), \ \xi\in\R^d, \ \Re\eta>0,$ has an analytic continuation to the region $|\Im\xi|< C_1\sqrt{\Re\eta/\La}, \ 0<\Re\eta< \La,$ and 
\begin{multline} \label{G3}
\|q(\xi',\eta')-q(\xi,\eta)\|\le C\La \left[ \ |\xi'-\xi|^\alpha +|(\eta'-\eta)/\La|^{\al/2} \ \right]  \ , \\
0<\Re\eta\le \Re\eta'\le \La, \quad \xi',\xi\in\C^d \ {\rm with \ } |\Im\xi|,  \ |\Im\xi'|\le C_1\sqrt{\Re\eta/\La} \ ,
\end{multline}
where $C$ is a constant depending on the environment and the function $\mathbf{a}(\cdot)$. 
\end{hypothesis}

Here we shall prove that Hypothesis 3.1 holds for the massless field theory environment $(\Om,\mathcal{F},P)$ of Theorem 1.1. To do this we recall some operators defined in \cite{cf1}. For any $g\in \mathcal{H}(\Om)$, let 
$\psi(\xi,\eta,\om)$ be  the solution to the equation
\be \label{H3}
\frac{1}{\La}[\eta+\pa]\psi(\xi,\eta,\om)+\pa_\xi^*\pa_\xi\psi(\xi,\eta,\om)=\pa^*_\xi g(\om), \quad \Re\eta>0, \ \xi\in \R^d, \ \om\in\Om.
\ee
The operator $T_{\xi,\eta}$ on $\mathcal{H}(\Om)$ is defined by  $T_{\xi,\eta} g(\cdot)=\pa_\xi\psi(\xi,\eta,\cdot)$. Let $G(x,t), \ x\in\Z^d, \ t>0,$ be the solution to the initial value problem
\begin{eqnarray} \label{I3}
\frac{\pa G(x,t)}{\pa t} +\na^*\na G(x,t) \ &=& \ 0, \quad x\in\Z^d, \ t>0, \\
G(x,0) \ &=& \ \del(x), \quad x\in\Z^d \ . \nonumber
\end{eqnarray}
It is well known that there exist positive constants $C,\ga$ depending only on $d$ such that $G$ satisfies the inequality
\begin{multline} \label{I*3}
G(x,t)+(t+1)^{1/2}|\na G(x,t)|+(t+1)|\na\na^* G(x,t)| \\ \le \ 
\ \frac{C}{[t+1]^{d/2}}\exp\left[-\ga\min\left\{ |x|, \ \frac{|x|^2}{t+1}\right\}\right] \ , \quad {\rm for \ } x\in\Z^d, \ t\ge 0.
\end{multline}
The operator $T_{\xi,\eta}$ is also given by the formula
\be \label{J3}
T_{\xi,\eta} g(\om) \ = \   \La\int_{0}^\infty e^{-\eta t} \ dt\sum_{x\in \Z^d} \left\{\nabla\nabla^* G(x,\La t)\right\}^*\exp[-ix.\xi]  \ g(\tau_{x,-t}\om) \ .
\ee 
 It easily follows from (\ref{H3}) that $T_{\xi,\eta}$ is a bounded operator on   $\mathcal{H}(\Om)$ with $\|T_{\xi,\eta} \|_{\mathcal{H}(\Om)}\le 1$  provided $\xi\in\R^d, \Re\eta>0$. Furthermore by Lemma 2.1 of \cite{cf1} the function $(\xi,\eta)\ra T_{\xi,\eta} $ from $\R^d\times\R$ to  the Banach space of bounded linear operators on $\mathcal{H}(\Om)$ has an analytic continuation to a strip $0<\Re\eta<\La, \ |\Im\xi |<C\sqrt{\Re\eta/\La}$ where $C$ is a constant depending only on $d$.   
 
 Let $\mathbf{b}$ be the $d\times d$ matrix valued function $ \mathbf{b}(\om)=I_d-\mathbf{a}(\om)/\La, \ \om\in\Om,$ whence (\ref{D1}) implies the quadratic form inequality $0\le \mathbf{b}(\cdot)\le (1-\la/\La)I_d$. It is easy to see from (\ref{C3}) that 
 \be \label{C*3}
q(\xi,\eta) \ = \  \langle \ {\bf a}(\cdot) \ \rangle- \La\sum_{m=1}^\infty\langle \ {\bf b}(\cdot)\left[ PT_{\xi,\eta}{\bf b}(\cdot)\right]^m \ \rangle  \ .
\ee
We consider $\xi\in\C^d,\eta\in\C$ with $\xi$ having fixed imaginary part,  $\eta$ having fixed positive real part, and satisfying the conditions of Hypothesis 3.1.  For $k=1,2,..,$ we define an operator $T_{k,\Im\xi,\Re\eta}$  from functions $g:\Z^{d}\times\R\ra\C^d\otimes \C^d$ to periodic functions $T_{k,\Im\xi,\Re\eta} \ g:[-\pi,\pi]^{d}\times\R\times\Om\ra \C^d\otimes\C^d$ by
\be \label{K3}
T_{k,\Im\xi,\Re\eta} \ g(\Re\xi,\Im\eta,\cdot) \ = \   \sum_{x\in\Z^{d}}\int_{-\infty}^\infty  dt \  g(x,t)\tau_{x,-t}\mathcal{P}{\bf b}(\cdot)\left[ \mathcal{P}T_{\xi,\eta}{\bf b}(\cdot)\right]^{k-1} \ ,
\ee
where $\xi=\Re\xi+i\Im\xi, \ \eta=\Re\eta+i\Im\eta$ in (\ref{K3}). 
For $1\le p<\infty$ let $L^p(\Z^d\times\R,\C^d\otimes\C^d)$ be the Banach space of $d\times d$ matrix valued functions $g:\Z^d\times\R\ra\C^d\otimes\C^d$ with norm $\|g\|_p$ defined by
\be \label{L3}
\|g\|_p^p  =  \sup_{v\in\C^d:|v|=1}\sum_{x\in\Z^d}\int_{-\infty}^\infty dt \  |g(x,t)v|_2^p \ ,
\ee
where $|g(x,t)v|_2$ is the Euclidean norm of the vector $g(x,t)v\in\C^d$. We similarly define the space  $L^\infty([-\pi,\pi]^d\times\R\times\Om,\C^d\otimes\C^d)$ of $d\times d$ matrix valued functions $g:[-\pi,\pi]^d\times\R\times\Om\ra \C^d\otimes\C^d$ with norm $\|g\|_\infty$ defined by
\be\label{M3}
\quad \|g\|_\infty= \sup_{v\in\C^d:|v|=1}\left[ \ \sup_{\zeta\in[-\pi,\pi]^d,  \theta\in\R}  ||g(\zeta,\theta,\cdot)v||_{\mathcal{H}(\Om)} \ \right]  \ .
\ee
Since $\|T_{\xi,\eta} \|_{\mathcal{H}(\Om)}\le 1$  if $\xi\in\R^d, \Re\eta>0$
 it follows from (\ref{K3}), (\ref{L3}) that if $\Im\xi=0$ then $T_{k,\Im\xi,\Re\eta}$  is a bounded operator from $L^1(\Z^d\times\R,\C^d\otimes\C^d)$ to $L^\infty([-\pi,\pi]^d\times\R\times\Om,\C^d\otimes\C^d)$  with norm  $\|T_{k,\Im\xi,\Re\eta}\|_{1,\infty}\le (1-\la/\La)^k$.  In the next section we show that  $T_{k,\Im\xi,\Re\eta}$ is a bounded operator from $L^p(\Z^d\times\R,\C^d\otimes\C^d)$ to $L^\infty([-\pi,\pi]^d\times\R\times\Om,\C^d\otimes\C^d)$  for some $p>1$  in the case of the environment of Theorem 1.1 and estimate its norm $\|T_{k,\Im\xi,\Re\eta}\|_{p,\infty}$.  In particular we prove:
\begin{lem}
Let $(\Om,\mathcal{F},P)$ be an environment of massless fields $\phi:\Z^d\times\R\ra\R$ with $d\ge 3$, and $\tilde{\mathbf{a}}:\R\ra\R^{d(d+1)/2}$ be  as in the statement of Theorem 1.1. Set  $\mathbf{a}(\phi)=\tilde{\mathbf{a}}(\phi(0,0)), \ \phi\in\Om$.  Then there exists $p_0(\La/\la)$ with $1<p_0(\La/\la)<2$ depending only on $d$ and $\La/\la$, and positive constants $C_1(\La/\la),C_2(\La/\la)$ depending only on $d$ and $\La/\la$ such that  for $ 0<\Re\eta<\La, \ \ |\Im\xi |<C_1(\La/\la)\sqrt{\Re\eta/\La}$,
\be \label{N3}
\|T_{k,\Im\xi,\Re\eta}\|_{p,\infty} \ \le \ \frac{C_2(\La/\la) k\|D\tilde{\mathbf{a}}(\cdot)\|_\infty}{\La^{5/2-1/p}}(1-\la/\La)^{(k-1)/2}  \big[1+C_2|\Im\xi|^2\big/(\Re\eta/\La)\ \big]^{k-1} 
\ee
 provided $ 1\le p\le p_0(\La/\la)$. 
\end{lem}
To complete this section we show how Lemma 3.1 implies that Hypothesis 3.1 holds.
\begin{proof}[Proof of Hypothesis 3.1]
We assume that $(\xi,\eta)$ and $(\xi',\eta')$ are as in the statement of  Hypothesis 3.1. Let $g:\Z^d\times\R\ra\C^d\otimes\C^d$ be the function defined by
 \be \label{O3}
 g(x,t) \ = \ \La\{\na\na^*G(x,\La t)\}^* \ e^{-ix\cdot\xi'-\eta't}-\La\{\na\na^*G(x,\La t)\}^* \ e^{-ix\cdot\xi-\eta t}  \ ,
 \ee
where the Green's function $G(\cdot,\cdot)$ is defined by (\ref{I3}). It follows  from (\ref{J3}), (\ref{C*3}) and Lemma 2.1 of \cite{cf1} that the constant $C_1>0$ in (\ref{G3}) can be chosen depending only on $d$ and $\La/\la$ so that
 \be \label{P3}
\|[q(\xi',\eta')-q(\xi,\eta)]v\| \ \le \  C_2\La\sum_{k=1}^\infty \|T_{k,\Im\xi,\Re\eta}g(\Re\xi,\Im\eta,\cdot)v\|_{\mathcal{H}(\Om)} \quad {\rm for \ } |\Im\xi|,  \ |\Im\xi'|\le C_1\sqrt{\Re\eta/\La} \ , 
\ee
where $C_2$ is a constant depending only on $d,\La/\la$. 
We can see from (\ref{I*3})  that there is a constant $C_1$ depending only on $d$ such that if $ |\Im\xi|,  \ |\Im\xi'|< C_1\sqrt{\Re\eta/\La}$ then  the function $g$ is in $L^p(\Z^{d}\times\R,\C^d\otimes\C^d)$ for any  $p>1$. Furthermore if $0\le \al\le 1$ and $p> (d+2)/(d+2-\al)$ then $\|g\|_p$ satisfies the inequality
 \be \label{Q3}
 \|g\|_p \ \le \ C_p\La^{1-1/p}[ \ |\xi'-\xi|^\al+|(\eta'-\eta)/\La|^{\al/2} \ ] \  ,
 \ee
 where the constant $C_p$ depends only on $d,p$. The H\"{o}lder continuity (\ref{G3}) for sufficiently small $\al>0$ follows from  (\ref{P3}), (\ref{Q3}) and Lemma 3.1.
\end{proof}

\vspace{.1in}

\section{Proof of Lemma 3.1}
In \cite{cf1}  we proved that the operator $T_{k,\Im\xi,\Re\eta}$ of (\ref{K3}) is for some $p$ in the range $1\le p\le p_0(\La/\la)$  a bounded operator from  $L^p(\Z^d\times\R,\C^d\otimes\C^d)$ to $L^\infty([-\pi,\pi]^d\times\R\times\Om,\C^d\otimes\C^d)$   if the environment $(\Om,\mathcal{F},P)$ is the stationary process for the SDE (\ref{O1}) with $m>0$.  Here we take an alternative approach to proving this result which will allow us to study the $m\ra 0$ limit of $\|T_{k,\Im\xi,\Re\eta}\|_{p,\infty}$. We  first establish an inequality  for periodic fields $\phi:Q\ra\R$ on cubes $Q\subset\Z^d$, and then show that we can let $Q\ra\Z^d$ since our estimates are independent of $Q$.  Let $L$ be an even integer and $Q=Q_L$ denote the lattice points of $\Z^d$ contained  in the cube of length $L$ centered at the origin. In the following we identify  all points $x,y\in Q$ with $x-y=L\mathbf{e}_k$ for some $k, \ 1\le k\le d$. 

As in \cite{cf1} the Malliavin calculus \cite{ct,nu} is the main tool we  use to prove Lemma 3.1. We assume that $V:\R^d\ra\R$ is a $C^2$ uniformly convex function such that $\mathbf{a}(\cdot)=V''(\cdot)$   satisfies (\ref{D1}) and $m>0$. Letting $B(x,\cdot), \ x\in Q,$ be  independent copies of Brownian motion, then  the SDE initial value problem
\begin{multline} \label{A4} 
d\phi(x,t) \ = \ -\frac{\pa}{\pa\phi(x,t)}\sum_{x'\in Q} \frac{1}{2}\{V(\na\phi(x',t))+m^2\phi(x',t)^2/2\} \ dt +dB(x,t) \\
 {\rm for \ } x\in Q, t>0,  \qquad  {\rm with \ }\phi(x,0) \ = \ 0  \ {\rm for \ }   x\in Q, 
\end{multline}
has a unique periodic solution $\phi(x,t), \ x\in Q,t>0$, which is continuous in $t\ge 0$ with probability $1$.  We denote  the function $\phi$ corresponding to a particular realization $\tilde{\om}$ of the white noise process $dB(\cdot,\cdot)$ as $\phi(\tilde{\om})$. Let $(\Om_{Q,{\rm Mal}}, \mathcal{F}_{Q,{\rm Mal}}, P_{Q,{\rm Mal}})$ be the Malliavin probability space associated with the Brownian motions $ B(x,\cdot), \ x\in Q$. We denote the Malliavin derivative of a function $G: \Om_{Q,{\rm Mal}}\ra\C$ at a point $\tilde{\om}\in \Om_{Q,{\rm Mal}}$ by $D_{\rm Mal}G(x,t;\tilde{\om}), \ x\in Q,t>0$. It is well known (see \cite{ct} Theorem 5.4) that the Poincar\'{e} inequality (\ref{J1})  holds for $(\Om_{Q,{\rm Mal}}, \mathcal{F}_{Q,{\rm Mal}}, P_{Q,{\rm Mal}})$ with constant $K_{\tilde{P}}=1$.  Thus we have that
\be \label{B4}
 {\rm Var}[ G(\cdot)] \ \le \  \langle \ \|D_{\rm Mal}G(\cdot;\tilde{\om})\|_2^2 \ \rangle_{\Om_{Q,{\rm Mal}}}
\ee
 where $\|\cdot\|_2$ is the Euclidean norm in $L^2(Q\times\R^+)$. 
 
 Let $\phi:\Z^d\times\R\ra\R$ be continuous and consider the terminal value problem for the backwards in time parabolic PDE
\begin{eqnarray} \label{C4}
\frac{\pa u(y,s)}{\pa s}  \ &=&  \ \frac{1}{2}\na^*V''(\na\phi(y,s))\na u(y,s), \quad    s<t,y\in \Z^d, \\
 u(y,t) \ &=& \ u_0(y),  \quad  y\in \Z^d, \nonumber
\end{eqnarray}
with solution
\be \label{D4}
u(y,s) \ = \ \sum_{x\in \Z^d} G(y,s;x,t,\phi) u_0(x) \ , \quad  s\le t,y\in \Z^d.
\ee
It is easy to see that if $u_0(\cdot)\in L^2(\Z^d)$ then $u(\cdot,s)\in L^2(\Z^d)$ for $s\le t$ and $\|u(\cdot,s)\|_{L^2(\Z^d)}\le \|u_0(\cdot)\|_{L^2(\Z^d)}$.
The function $G$ satisfies
\be \label{E4}
\sum_{x\in \Z^d} G(y,s;x,t,\phi) \ = \ 1  \ {\rm for \ } y\in \Z^d, \quad \sum_{y\in \Z^d}G(y,s;x,t,\phi) \ = \ 1 \ {\rm for \ } x\in \Z^d, 
\ee
and is non-negative if $V''(\cdot)$ is diagonal.  In that case the function $x\ra G(y,s;x,t,\phi), \ x\in \Z^d,$ is the pdf for the position at time $t$ of a continuous time random walk started at $y$ at time $s$. If $\phi:Q\times\R\ra\R$ is periodic we can extend it to a periodic function $\phi:\Z^d\times\R\ra\R$. Let $u_0:Q\ra\R$ be periodic and extend it to a periodic function $u_0:\Z^d\ra\R$. Then the solution to the periodic terminal value problem (\ref{C4}) is given by
\be \label{F4}
u(y,s) \ = \ \sum_{x\in Q} G_Q(y,s;x,t,\phi) u_0(x) \ , \quad  s\le t,y\in Q,
\ee
where $G_Q$ is the periodic Green's function
\be \label{G4}
G_Q(y,s;x,t) \ = \ \sum_{n\in\Z^d} G(y,s;x+Ln,t,\phi) \ , \quad s\le t, \ x,y\in\Z^d.
\ee
 It was shown in  \cite{cf1} that the Malliavin derivative of $\phi(x,t,\tilde{\om}), \ x\in Q,t>0,$ is  given by the formula
\begin{eqnarray} \label{H4}
D_{\rm Mal}\phi(y,s;x,t,\tilde{\om}) \ &=& \  e^{-m^2(t-s)/2}G_Q(y,s;x,t,\phi(\tilde{\om})) \quad {\rm for \ } 0<s<t, \\
D_{\rm Mal}\phi(y,s;x,t,\tilde{\om}) \ &=& \  0 \quad {\rm for \ } s>t, \quad y\in Q. \nonumber
\end{eqnarray}

The solution to (\ref{C4}) can be written in a perturbation expansion by setting $V''(z)=\La [I_d-\tilde{\mathbf{b}}_V(z)], \ z\in\R^d,$  where $0\le \tilde{\mathbf{b}}_V(\cdot)\le (1-\la/\La)I_d$ in the quadratic form sense.  Then
\be \label{I4}
u(y,s) \ = \ \sum_{n=0}^\infty u_n(y,s), \quad y\in \Z^d,s<t,
\ee
where $u_0(y,s)$ is the solution to the terminal value problem
\begin{eqnarray} \label{J4}
\frac{\pa u_0(y,s)}{\pa s}  \ &=&  \ \frac{\La}{2}\na^*\na u_0(y,s), \quad    s<t,y\in \Z^d, \\
 u_0(y,t) \ &=& \ u_0(y),  \quad  y\in \Z^d, \nonumber
\end{eqnarray}
and the $u_n(y,s), \ n=1,2,..,$ solutions to the terminal value problems
\begin{multline} \label{K4}
\frac{\pa u_n(y,s)}{\pa s}  \ =  \ \frac{\La}{2}\left[\na^*\na u_n(y,s) - \na^*\tilde{\mathbf{b}}_V(\na\phi(y,s))\na u_{n-1}(y,s)\right], \quad    s<t,y\in \Z^d,\\
 u_n(y,t) \ = \ 0,  \quad  y\in \Z^d. 
\end{multline}
It follows from (\ref{I3}) that
\be \label{L4}
u_0(y,s) \ = \ \sum_{z\in\Z^d}  G\big(y-z,\La(t-s)/2\big) u_0(z), \quad y\in \Z^d,s<t.
\ee 
Similarly we have that for $n\ge 1$,
\be \label{M4}
u_n(z,r) \ = \ \sum_{y\in\Z^d}\int_{0}^{t-r} ds \  \na  G\big(y,\La s/2\big) \tilde{\mathbf{b}}_V(\na\phi(z+y,r+s))\na u_{n-1}(z+y,r+s)\, \quad z\in \Z^d,r<t.
\ee
If we set $u_0$ in (\ref{C4}) to be given by $u_0(y)=\del(y-x), \ y\in\Z^d,$ then the perturbation expansion (\ref{I4}) yields a perturbation expansion for the Green's function,
\be \label{N4}
G(y,s;x,t,\phi) \ = \ \sum_{n=0}^\infty G_n(y,s;x,t,\phi) \ ,
\ee
where the $G_n$ are multilinear in $\tilde{\mathbf{b}}_V$ of degree $n$.   By choosing $u_0(y)=\sum_{n\in\Z^d} \del(y-x-nL), \ y\in\Z^d,$ we obtain a similar perturbation expansion for the periodic Green's function
\be \label{O4}
G_Q(y,s;x,t,\phi) \ = \ \sum_{n=0}^\infty G_{n,Q}(y,s;x,t,\phi) \ .
\ee

Next we consider the inhomogeneous problem
\begin{eqnarray} \label{P4}
\frac{\pa u(y,s)}{\pa s}  \ &=&  \  \frac{1}{2}\na^*V''(\na\phi(y,s))\na u(y,s)-f(y,s), \quad    s\in\R,y\in \Z^d, \\
 \lim_{s\ra +\infty} u(y,s) \ &=& \ 0,  \quad  y\in \Z^d. \nonumber
\end{eqnarray}
Let $f:\Z^d\times\R\ra\C$ be a continuous function such that  $\int_{-\infty}^\infty  dt \ \|f(\cdot,t)\|_{L^2(\Z^d)}<\infty$.  From  Duhamel's formula we see that the solution to (\ref{P4}) is given by
\be \label{Q4}
u(y,s) \ = \ \sum_{x\in \Z^d} \int_s^\infty dt \ G(y,s;x,t,\phi) f(x,t) \ , \quad  s\in\R,y\in \Z^d.
\ee
We can similarly consider the inhomogeneous periodic problem where $f:Q\times\R\ra\C$ is assumed periodic and we extend it to a periodic function $f:\Z^d\times\R\ra\C$. The solution to (\ref{P4}) is then periodic and is  given by the formula
\be \label{R4}
u(y,s) \ = \ \sum_{x\in Q} \int_s^\infty dt \ G_Q(y,s;x,t,\phi) f(x,t) \ , \quad  s\in\R, y\in Q.
\ee
\begin{lem}
Assume $g:\Z^d\times\R\ra\C$ is in  $ L^2(\Z^d\times\R)$  and define for $m>0$ the function  $v:\Z^d\times\R\ra\C$  by
\be \label{S4}
v(y,s) \ = \ \sum_{x\in \Z^d} \int_s^\infty dt \ e^{-m^2(t-s)/2}G(y,s;x,t,\phi) g(x,t) \ , \quad  s\in\R,y\in \Z^d.
\ee
Then $v$ is also in $L^2(\Z^d\times\R)$ and $\|v(\cdot,\cdot)\|_{L^2(\Z^d\times\R)}\le 2m^{-2}\|g(\cdot,\cdot)\|_{L^2(\Z^d\times\R)}$. Corresponding to the perturbation expansion  (\ref{N4}) the function $v$ can be written as a sum
\be \label{W4}
v \ = \ \sum_{n=0}^\infty v_n \quad {\rm where \ } \|v_n(\cdot,\cdot)\|_{L^2(\Z^d\times\R)}\le \frac{1}{m^2}(1-\la/\La)^n \|g(\cdot,\cdot)\|_{L^2(\Z^d\times\R)} \ .
\ee
\end{lem}
 \begin{proof}
 It follows from (\ref{P4}) that $v$ satisfies
 \be \label{T4}
\left[\frac{\pa}{\pa s}  -\frac{m^2}{2}\right] v(y,s) \ =  \  \frac{1}{2}\na^*V''(\na\phi(y,s))\na v(y,s)-g(y,s).
\ee
Multiplying (\ref{T4}) by $\overline{v(y,s)}$, summing over $y\in\Z^d$ and integrating with respect to $s$ in the interval $-T\le s\le T$ we see  that
\be \label{U4}
\frac{m^2}{2}\sum_{y\in\Z^d} \int_{-T}^T ds \ |v(y,s)|^2 \ \le \ \frac{1}{2}\sum_{y\in\Z^d} \left\{|v(y,T)|^2-|v(y,-T)|^2\right\}+ \Re\left[ \sum_{y\in\Z^d} \int_{-T}^T ds \  \overline{v(y,s)} g(y,s) \ \right] \ .
\ee
Now from (\ref{S4}) we have that for $s\in\R$,
\begin{multline} \label{V4}
\|v(\cdot,s)\|_{L^2(\Z^d)} \ \le  \int_{0}^\infty e^{-m^2t/2}\|g(\cdot,s+t)\|_{L^2(\Z^d)} \  dt \\
\le \ \frac{1}{m} \left[ \int_0^\infty dt \ \|g(\cdot,s+t)\|^2_{L^2(\Z^d)} \ \right]^{1/2} \ .
\end{multline}
Since $g\in  L^2(\Z^d\times\R)$ it follows that the last expression on the RHS of (\ref{V4}) vanishes as $s\ra\infty$, whence $\lim_{T\ra\infty} \|v(\cdot,T)\|_{L^2(\Z^d)}=0$. We similarly conclude that 
$\lim_{T\ra\infty} \|v(\cdot,-T)\|_{L^2(\Z^d)}=0$.  Letting $T\ra\infty$ in (\ref{U4}) and using the Schwarz inequality we see that  if $g\in L^2(\Z^d\times\R)$ then $v$ is also in $L^2(\Z^d\times\R)$  and their norms are related as stated. 

To prove (\ref{W4}) we observe that
\be \label{X4}
\left[\frac{\pa}{\pa s}  -\frac{m^2}{2}\right] v_0(y,s) \ =  \  \frac{\La}{2}\na^*\na v_0(y,s)-g(y,s),
\ee
and that for $n\ge1$,
\be \label{Y4}
\left[\frac{\pa}{\pa s}  -\frac{m^2}{2}\right] v_n(y,s) \ =  \  \frac{\La}{2}\left[\na^*\na v_n(y,s)- \na^*\tilde{\mathbf{b}}_V(\na\phi(y,s))\na v_{n-1}(y,s)\right] \ .
\ee
Arguing as in the previous paragraph we see from (\ref{X4})  that
\be \label{Z4}
\frac{m^2}{4} \|v_0(\cdot,\cdot)\|^2_{L^2(\Z^d\times\R)}+ \frac{\La}{2}\|\na v_0(\cdot,\cdot)\|^2_{L^2(\Z^d\times\R,\C^d)} \ \le \ \frac{1}{m^2} \|g(\cdot,\cdot)\|^2_{L^2(\Z^d\times\R)} \ .
\ee
Similarly we have from (\ref{Y4}) that for $n\ge 1$,
\be \label{AA4}
\frac{m^2}{2} \|v_n(\cdot,\cdot)\|^2_{L^2(\Z^d\times\R)}+ \frac{\La}{4}\|\na v_n(\cdot,\cdot)\|^2_{L^2(\Z^d\times\R,\C^d)} \ \le  \  \frac{\La}{4}(1-\la/\La)^2\|\na v_{n-1}(\cdot,\cdot)\|^2_{L^2(\Z^d\times\R,\C^d)}  \  .
\ee
The inequality in (\ref{W4}) easily follows from (\ref{Z4}), (\ref{AA4}). 
 \end{proof}
 \begin{remark}
The result of Lemma 4.1 holds if $\Z^d$ is replaced by a periodic cube $Q$ with the Green's function $G$ replaced by the periodic Green's function $G_Q$ of (\ref{G4}) with perturbation expansion (\ref{O4}).  
 \end{remark}
 In \cite{cf1} we considered vector valued functions $F(\phi)$ of fields $\phi:\Z^d\times\R\ra\R$ and defined the {\it field derivative}  of $F$ at $\phi$ to be the function $d F(\cdot,\cdot;\phi)$ with domain $\Z^d\times\R$ which satisfies
 \be \label{AB4}
\lim_{\ve\ra 0} [F(\phi+\ve h)-F(\phi)]/\ve \ = \ \sum_{y\in\Z^d}  \int_{-\infty}^\infty ds  \ dF(y,s;\phi) h(y,s)
 \ee
 for all continuous functions $h:\Z^d\times\R\ra\R$ with compact support. 
Let  $\tilde{\mathbf{b}}:\R\ra\R^{d(d+1)/2}$ be a $C^1$ function taking values in the symmetric $d\times d$ matrices such that $\|\tilde{\mathbf{b}}\|_\infty+\|D\tilde{\mathbf{b}}\|_\infty<\infty$ and define $\mathbf{b}(\cdot)$ in (\ref{K3}) as a function of fields $\phi:\Z^d\times\R\ra\R$ by setting  $\mathbf{b}(\phi)=\tilde{\mathbf{b}}(\phi(0,0))$. It follows from (\ref{J3}), (\ref{K3}) that for $v\in\C^d$ one has  $T_{k,\Im\xi,\Re\eta} \ g(\Re\xi,\Im\eta,\phi)v \ = \  F(\phi)$ where
\begin{multline} \label{AC4}
F(\phi) \ = \  \La^{k-1}\left\{\prod_{j=0}^{k-1} \sum_{x_j\in\Z^{d}}\int_{-\infty}^\infty  dt_j \right\} \  \exp[-\eta (t_{k-1}-t_0)-i(x_{k-1}-x_0)\cdot\xi] \  g(x_0,t_0)  \\
\prod_{j=1}^{k-1} \left\{\nabla\nabla^* G(x_j-x_{j-1},\La (t_j-t_{j-1}))\right\}^* 
\mathcal{P}\tilde{\mathbf{b}}(\phi(x_0,-t_0)) \
\mathcal{P}\tilde{\mathbf{b}}(\phi(x_1,-t_1))\cdots \mathcal{P}\tilde{\mathbf{b}}(\phi(x_{k-1},-t_{k-1}))v \ .
\end{multline}
In (\ref{AC4}) we have extended the domain of the Green's function $G(x,t)$ defined by  (\ref{I3}) for $t\ge 0$ to $t<0$ by setting $G(x,t)=0, \ t<0$. The operator $I-\mathcal{P}$ is now any linear operator taking $d\times d$ symmetric matrix valued functions $\mathbf{b}(\cdot)$ of $\phi:\Z^d\times\R\ra\R$ to a constant  matrix which has the property that $\|(I-\mathcal{P})\mathbf{b}\|_\infty\le  \sup_{\phi:\Z^d\times\R\ra\R}\|\mathbf{b}(\phi)\|_\infty$.
We can see from (\ref{I*3}) that there exists $C_1>0$ depending only on $d$ such that if $(\xi,\eta)$ satisfies the inequality   
\be \label{AD4}
0<\Re\eta<\La, \quad |\Im\xi|<  C_1\sqrt{\Re\eta/\La} \ ,
\ee
and $g\in L^1(\Z^d\times\R,\C^d\otimes\C^d)$ then $F(\phi)\in\C^d$ is bounded by
\be \label{AD*4} 
|F(\phi)|\le C\La^{k-1}\|\tilde{\mathbf{b}}\|^{k}_\infty[\Re\eta]^{-(k-1)}\|g\|_1|v|,
\ee
where the constant $C$ depends only on $d,k$.   Furthermore $F$ is differentiable in the sense of (\ref{AB4}) and the field derivative is given by the formula
\begin{multline} \label{AE4}
d F(y,s;\phi) \ = \  \La^{k-1}\left\{\prod_{j=0}^{k-1} \sum_{x_j\in\Z^{d}}\int_{-\infty}^\infty  dt_j \right\} \  \exp[-\eta (t_{k-1}-t_0)-i(x_{k-1}-x_0)\cdot\xi] \  g(x_0,t_0)  \\
\prod_{j=1}^{k-1} \left\{\nabla\nabla^* G(x_j-x_{j-1},\La (t_j-t_{j-1}))\right\}^* \\
\bigg[\del(x_0-y,t_0+s) 
D\tilde{\mathbf{b}}(\phi(x_0,-t_0)) \
\mathcal{P}\tilde{\mathbf{b}}(\phi(x_1,-t_1))\cdots \mathcal{P}\tilde{\mathbf{b}}(\phi(x_{k-1},-t_{k-1}))v+\cdots + \\
\del(x_{k-1}-y,t_{k-1}+s) 
\tilde{\mathbf{b}}(\phi(x_0,-t_0)) \
\tilde{\mathbf{b}}(\phi(x_1,-t_1))\cdots D\tilde{\mathbf{b}}(\phi(x_{k-1},-t_{k-1}))v \ \bigg]
\end{multline}
for $(y,s)\in\Z^d\times\R$, where $\del(x,t)=\del(x)\del(t)$ is the product of the Kronecker and Dirac delta functions.  If $g$ is also in $ L^2(\Z^d\times\R,\C^d\otimes\C^d)$ then $dF(\cdot,\cdot;\phi)\in L^2(\Z^d\times\R,\C^d)$ and from (\ref{I*3}) we see that
\begin{multline} \label{AF4}
\|dF(\cdot,\cdot;\phi)\|_2  \ \le \ C\La^{k-1}\|D\tilde{\mathbf{b}}\|_\infty \|\tilde{\mathbf{b}}\|^{k-1}_\infty[\Re\eta]^{-(k-1)} \\
\left[ \ \|g\|_2+\Re\eta \  \|g\|_1\left\{\sum_{x\in\Z^d}\int_0^\infty dt \  |\na\na^*G(x,\La t)|^2 \right\}^{1/2}\ \right]|v| \ ,
\end{multline} 
where $C$ depends only on $d,k$. 
\begin{lem}
Let $(\Om,\mathcal{F},P_m)$ with $m>0$ be the environment of massive fields $\phi:\Z^d\times\R\ra\R$ defined by (\ref{O1}), (\ref{P1}) and $g:\Z^d\times\R\ra\C^d\otimes\C^d$ a continuous function of compact support.  Then there exists $C_1>0$ depending only on $d$ such that if $(\xi,\eta)$ lies in the region (\ref{AD4}) the operator of (\ref{K3}) satisfies the inequality
\begin{multline} \label{AG4}
\left\langle \ \left|T_{k,\Im\xi,\Re\eta} \ g(\Re\xi,\Im\eta,\cdot) v\right|^2 \   \right\rangle \ \le  \\
\left\langle \ \sum_{y\in\Z^d}\int_{-\infty}^\infty ds \ \left|\sum_{x\in \Z^d} \int_s^\infty dt \ e^{-m^2(t-s)/2}G(y,s;x,t,\phi) d F(x,t;\phi)\right|^2 \ \right\rangle \ ,
\end{multline}
where $G$ is the Green's function defined by (\ref{D4}), and $d F$ is the field derivative (\ref{AE4}). 
\end{lem}
\begin{proof}
It follows from (\ref{AD*4}) that the LHS of (\ref{AG4}) is the expectation of a bounded function. From Lemma 4.1 and (\ref{AF4}) we see that the RHS is also the expectation of a bounded function.  To prove (\ref{AG4}) we use the Poincar\'{e} inequality (\ref{B4}) and the formula (\ref{H4}) for the Malliavin derivative.  Thus let $Q\subset\Z^d$ be the periodic cube with side of length $L$ and $\phi(x,t),  \ x\in Q,t\ge 0,$ the solution to the initial value problem (\ref{A4}).  For $T>0$ we denote by $\phi_T$ the periodic field $\phi_T:Q\times\R\ra\R$ defined by $\phi_T(x,t)=\phi(x,T+t), \ x\in Q,t\ge -T$ where $\phi$ is the solution to (\ref{A4}), and $\phi_T(x,t)=0, \ x\in Q,t<-T$.  We extend the field $\phi_T$ to  a periodic field $\phi_T:\Z^d\times\R\ra\R$. From (\ref{H4}) we have that for the function $F$ of (\ref{AC4}) if $y\in Q$ and $s>-T$ then
\be \label{AH4}
D_{\rm Mal} F(y,T+s;\phi_T) \ = \ \sum_{x\in Q} \int_s^\infty dt \ e^{-m^2(t-s)/2}G_Q(y,s;x,t,\phi_T) d F_Q(x,t;\phi_T) \ ,
\ee
where $dF_Q$ is given in terms of (\ref{AE4}) by
\be \label{AN4}
dF_Q(x,t;\phi) \ = \ \sum_{n\in\Z^d} dF(x+Ln,t;\phi) \ .
\ee
It is easy to see from (\ref{AE4}) that $\|dF_Q(\cdot,\cdot;\phi)\|_\infty\le C$ for some constant independent of $\phi$ so the RHS of (\ref{AH4}) is bounded. 

The invariant measure associated with the Markov process defined by the SDE (\ref{A4})  is given by the formula
\be \label{AI4}
\exp \left[ - \sum_{x\in Q} V\left( \na\phi(x)\right)+m^2\phi(x)^2/2 \right] \prod_{x\in Q} d\phi(x)/{\rm normalization}.
\ee
 We denote the probability space for the corresponding stationary process $\phi(x,t), \ x\in Q,t\in\R$ by $(\Om_Q,\mathcal{F}_Q,P_{Q,m})$ and expectation with respect to the measure $P_{Q,m}$ by $\langle\cdot\rangle_{\Om_{Q,m}}$. Evidently $(\Om_Q,\mathcal{F}_Q,P_{Q,m})$is invariant with respect to the translation operators $\tau_{x,t}:\Om_Q\ra\Om_Q, \ x\in\Z^d,t\in\R$. Our first goal will be to obtain a version of the inequality (\ref{AG4}) for the operator $T_{k,\Im\xi,\Re\eta}$ of (\ref{K3}) when the random environment is given by $(\Om_Q,\mathcal{F}_Q,P_{Q,m})$. To do this we use the fact (see Appendix A for a proof) that for any $N\ge 1$, continuous bounded function $f:\R^{N}\ra\C$, and $(x_1,t_1),...,(x_N,t_N)\in Q\times\R$,
\be \label{AJ4}
\lim_{T\ra\infty} \left\langle \ f\big( \ \phi_T(x_1,t_1),....,\phi_T(x_N,t_N) \ \big) \ \right\rangle_{\Om_{Q,{\rm Mal}}} \ = \ \left\langle \ f\big( \ \phi(x_1,t_1),....,\phi(x_N,t_N) \ \big) \ \right\rangle_{\Om_{Q,m}} \ .
\ee
It follows from (\ref{AJ4}) that 
\be \label{AK4}
\lim_{T\ra\infty} \left\langle \ v^*\prod_{j=1}^N\mathcal{P}\tilde{\mathbf{b}}(\phi_T(x_j,-t_j)) \ v \ \right\rangle_{\Om_{Q,{\rm Mal}}}  \ = \ 
  \left\langle  \ v^*\prod_{j=1}^N\mathcal{P}\tilde{\mathbf{b}}(\phi(x_j,-t_j)) \ v   \ \right\rangle_{\Om_{Q,m}} \ ,
\ee
where $I-\mathcal{P}$ on the LHS of (\ref{AK4}) denotes  expectation with respect to $\langle\cdot\rangle_{\Om_{Q,{\rm Mal}}}  $ and on the RHS expectation with respect to  $\langle\cdot\rangle_{\Om_{Q,m}}$. We conclude from (\ref{AK4}), Fubini's theorem and the dominated convergence theorem that
\be \label{AL4}
\lim_{T\ra\infty} \left\langle \ |F(\phi_T)|^2  \ \right\rangle_{\Om_{Q,{\rm Mal}}}  \ = \ 
 \left\langle  \ |F(\phi)|^2  \ \right\rangle_{\Om_{Q,m}}  \ = \ \left\langle \ \left|T_{k,\Im\xi,\Re\eta} \ g(\Re\xi,\Im\eta,\cdot) v\right|^2 \   \right\rangle_{\Om_{Q,m}} \ . 
\ee

Next we see from the Poincar\'{e} inequality (\ref{B4}) and (\ref{AH4}) that
\be \label{AM4}
\left\langle \ |F(\phi_T)|^2  \ \right\rangle_{\Om_{Q,{\rm Mal}}}  \ \le \ 
\left\langle \  \sum_{y\in\Z^d}\int_{-T}^\infty ds \ \left|\sum_{x\in \Z^d} \int_s^\infty dt \ e^{-m^2(t-s)/2}G_Q(y,s;x,t,\phi_T) d F_Q(x,t;\phi_T)\right|^2 \ \right\rangle_{\Om_{Q,{\rm Mal}}}   \ .
\ee
We assume that $L$ is sufficiently large so that the support of $g$ is contained in $Q\times\R$. It is easy to see then that $dF_Q(y,s;\phi)$ is given by the RHS of (\ref{AE4}) with $\Z^d$ replaced by $Q$ and the function $G(x,\La t)$ replaced by the corresponding periodic Green's function on $Q$. Hence $\|dF_Q(\cdot,\cdot,\phi)\|_{L^2(Q\times\R,\C^d)}$ is bounded by the periodic version of the RHS of (\ref{AF4}).  From this we see that  $\limsup_{Q\ra\Z^d}\|dF_Q(\cdot,\cdot,\phi)\|_{L^2(Q\times\R,\C^d)}<\infty$. Observe now that we can argue as in the proof of (\ref{AL4}) to conclude that if $G_{n,Q}, \ n=0,1,2,..,$ denote the terms in the perturbation expansion (\ref{O4}) and $N\ge 0$ then 
\begin{multline} \label{AO4}
 \lim_{T\ra\infty}\left\langle \   \sum_{y\in\Z^d}\int_{-T}^\infty ds \ \left|\sum_{x\in \Z^d} \int_s^\infty dt \ e^{-m^2(t-s)/2}\sum_{n=0}^NG_{n,Q}(y,s;x,t,\phi_T) d F_Q(x,t;\phi_T)\right|^2  \ \right\rangle_{\Om_{Q,{\rm Mal}}}  \\
  = \ 
 \left\langle  \  \sum_{y\in\Z^d}\int_{-\infty}^\infty ds \ \left|\sum_{x\in \Z^d} \int_s^\infty dt \ e^{-m^2(t-s)/2}\sum_{n=0}^NG_{n,Q}(y,s;x,t,\phi) d F_Q(x,t;\phi)\right|^2   \ \right\rangle_{\Om_{Q,m}}  \ .
\end{multline}
It follows from (\ref{AL4}), (\ref{AO4}) and the periodic version of Lemma 4.1 that
\begin{multline} \label{AP4}
\left\langle \ \left|T_{k,\Im\xi,\Re\eta} \ g(\Re\xi,\Im\eta,\cdot) v\right|^2 \   \right\rangle_{\Om_{Q,m}}  \ \le  \\
\left\langle \ \sum_{y\in\Z^d}\int_{-\infty}^\infty ds \ \left|\sum_{x\in \Z^d} \int_s^\infty dt \ e^{-m^2(t-s)/2}G_Q(y,s;x,t,\phi) d F_Q(x,t;\phi)\right|^2 \ \right\rangle_{\Om_{Q,m}}  \ .
\end{multline}

Finally we let $Q\ra\Z^d$ in (\ref{AP4}) to obtain (\ref{AG4}).  We denote by $\langle\cdot\rangle_{\Om_m}$ expectation with respect to the stationary process defined by (\ref{O1}), (\ref{P1}). It was proved in \cite{fs}) (see also \cite{c1}) that for any $N\ge 1$, continuous bounded function $f:\R^{N}\ra\C$, and $(x_1,t_1),...,(x_N,t_N)\in Q\times\R$,
\be \label{AQ4}
\lim_{Q\ra\Z^d} \left\langle \ f\big( \ \phi(x_1,t_1),....,\phi(x_N,t_N) \ \big) \ \right\rangle_{\Om_{Q,m}} \ = \ \left\langle \ f\big( \ \phi(x_1,t_1),....,\phi(x_N,t_N) \ \big) \ \right\rangle_{\Om_{m}} \ .
\ee
Hence we can using (\ref{AQ4}) argue as with the $T\ra\infty$ limit to conclude that (\ref{AG4}) holds for any $m>0$. 
\end{proof}
For  $\xi\in\R^d$ and $u:\Z^d\ra\C$ we denote by $\na_\xi u:\Z^d\ra\C^d$ the column vector $\na_\xi u(z)=[\na_{1,\xi}u(z),...,\na_{j,\xi}u(z)], \ z\in \Z^d,$ where $\na_{j,\xi}u(z)=e^{-i\mathbf{e}_j\cdot\xi}u(z+\mathbf{e}_j)-u(z), \ z\in \Z^d, \ j=1,..,d$. The column operator $\na_\xi$ has adjoint $\na^*_\xi$ which is a row operator. If $f\in  L^2(\Z^d\times\R,\C^d)$ and  $\xi\in\R^d,\eta\in\C$ with $\Re\eta>0$ there is a unique solution  $u(\xi,\eta,\cdot)\in L^2(\Z^d\times\R)$ to the PDE
\be \label{AR4}
\frac{1}{\La}\left[\eta-\frac{\pa}{\pa r}\right] \ u(\xi,\eta,z,r)+\na_\xi^*\na_\xi u(\xi,\eta,z,r) \ = \  \na^*_\xi f(z,r), \quad z\in \Z^d,r\in\R. 
\ee
Furthermore we have that 
\be \label{AS4}
\frac{\Re\eta}{\La}\|u(\xi,\eta,\cdot)\|^2_{L^2(\Z^d\times\R)}+\|\na_\xi u(\xi,\eta,\cdot)\|^2_{L^2(\Z^d\times\R,\C^d)} \ \le  \ \|f(\cdot)\|^2_{L^2(\Z^d\times\R,\C^d)} \ .
\ee
We also see similarly to (\ref{J3}) that $\na_\xi u(\xi,\eta,\cdot)=\tilde{T}_{\xi,\eta} f(\cdot)$ where
\be \label{AT4}
\tilde{T}_{\xi,\eta} f(z,r) \ = \   \La\int_{0}^\infty e^{-\eta s} \ ds\sum_{y\in \Z^d} \left\{\nabla\nabla^* G(y,\La s)\right\}^*\exp[-iy.\xi]  \ f(z+y,r+s) \ , \quad z\in \Z^d,r\in\R.
\ee 
 From Lemma 2.1 of \cite{cf1} we have that $u$ regarded as a function $(\xi,\eta)\ra L^2(\Z^d\times\R)$ has an analytic continuation to the region  (\ref{AD4}) with $C_1>0$  depending only on $d$. For $(\xi,\eta)$ in this region there is a constant $C_2$ depending only on $d$ such that 
\be \label{AU4}
\frac{\Re\eta}{2\La}\|u(\xi,\eta,\cdot)\|^2_{L^2(\Z^d\times\R)}+\|\na_\xi u(\xi,\eta,\cdot)\|^2_{L^2(\Z^d\times\R,\C^d)} \ \le \  \big[1+C_2|\Im\xi|^2\big/(\Re\eta/\La)\ \big] \|f(\cdot)\|^2_{L^2(\Z^d\times\R,\C^d)} \ .
\ee

For $\phi:\Z^d\times\R\ra\R$ a continuous function we  extend the corresponding Green's function $G(y,s;x,t,\phi)$ defined by (\ref{D4}) for $s\le t$ to $s>t$ by setting  $G(y,s;x,t,\phi)=0$ when $s>t$. For  $\xi\in\R^d,\eta\in\C,v\in \C^d,$ with $\Re\eta>0$ let $f_{2}(\xi,\eta,z,r,\phi), \   z\in \Z^d,r\in\R$ with range in $\C^d$ be defined by
\be \label{AV4}
f_{2}(\xi,\eta,z,r,\phi) \ = \  D\tilde{{\bf b}}(\phi(z,-r))v  \  e^{m^2r/2}G(0,0;z,-r,\phi)  \ .
\ee
It is evident from (\ref{C4}), (\ref{D4}) that $f_{2}(\xi,\eta,\cdot,\cdot,\phi)$ is in $L^2(\Z^d\times\R)$, whence we can define the function $u_{2}(\xi,\eta,z,r,\phi)$ as the solution to (\ref{AR4}) with $f(\cdot,\cdot)=f_{2}(\xi,\eta,\cdot,\cdot,\phi)$. Let $\mathbf{b}(\phi)=\tilde{\mathbf{b}}(\phi(0,0))$ and for $k=2,3,..,$ set
\be \label{AW4}
\pa_\xi F_{k}(\xi,\eta,\phi) \ = \ [\mathcal{P}T_{\xi,\eta}\mathbf{b}(\cdot)]^{k-1}v \ ,
\ee
where $T_{\xi,\eta}$ is the operator (\ref{J3}). 
We then inductively define functions $f_{k}, \ u_{k}$ for $k=3,4..,$  by the formula
\begin{multline} \label{AX4}
f_{k}(\xi,\eta,z,r,\phi) \ = \  D\tilde{{\bf b}}(\phi(z,-r))\pa_\xi F_{k-1}(\xi,\eta,\tau_{z,-r}\phi) \   e^{m^2r/2}G(0,0;z,-r,\phi)\\
+\tilde{{\bf b}}(\phi(z,-r)\na_\xi u_{k-1}(\xi,\eta,z,r,\phi) \ ,
\end{multline}
where  for $k=3,4,..,$ the function $u_{k}(\xi,\eta,z,r,\phi)$ is the solution to (\ref{AR4}) with $f(\cdot,\cdot)=f_{k}(\xi,\eta,\cdot,\cdot,\phi)$. 
The $u_{k}(\xi,\eta,z,r,\phi)$ and   $f_{k}(\xi,\eta,z,r,\phi)$  for $\xi\in\C^d$ are defined by analytic continuation from their values when $\xi\in\R^d$.
\begin{lem}
Let $G$ be the Green's function defined by (\ref{D4}) and for any $k\ge 2,v\in\C^d,$ let $dF$ be the function (\ref{AE4}). Then there exists $C_1>0$ depending only on $d$ such that for $(\xi,\eta)$ in the region (\ref{AD4}) the following identity  holds:
\begin{multline} \label{AY4}
\sum_{y\in \Z^d} \int_{-r}^\infty ds \ e^{-m^2(s+r)/2}G(z,-r;y,s,\tau_{-z,r}\phi) d F(y,s;\tau_{-z,r}\phi) \ = \\
\sum_{x\in\Z^{d}}\int_{-\infty}^\infty  dt \  g(x,t) \ e^{m^2(t-r)/2}G(0,0;x-z,r-t,\phi)D\tilde{\mathbf{b}}(\phi(x-z,r-t))\pa_\xi F_k(\xi,\eta,\tau_{x-z,r-t}\phi) \\
+
\sum_{x\in \Z^d} \int_{-\infty}^\infty dt \  g(x,t)\tilde{\mathbf{b}}(\phi(x-z,r-t))\na_\xi u_{k}(\xi,\eta,x-z,t-r,\phi) \ .
\end{multline}
\end{lem} 
\begin{proof}
We note that the first term on the RHS of (\ref{AY4}) comes from the sum which contains $\del(x_0-y,t_0+s)$ on the RHS of ( \ref{AE4}), in which we make the change of variables $(y,s)\leftrightarrow (x,-t)$. The remaining part of the RHS of (\ref{AE4}) is the same as
\begin{multline} \label{AZ4}
\sum_{x\in \Z^d} \int_{-\infty}^\infty dt \  g(x,t)\tilde{\mathbf{b}}(\phi(x-z,r-t))\\
\sum_{y\in \Z^d} \int_{-r}^\infty ds \ e^{-m^2(s+r)/2}G(z,-r;y,s,\tau_{-z,r}\phi) d H_k(y-x,s+t;\tau_{x-z,r-t}\phi) \ ,
\end{multline}
where $dH_k$ is the field derivative of the function
\be \label{BA4}
H_k(\phi) \ =  \ [\mathcal{P}T_{\xi,\eta}\mathbf{b}(\cdot)]^{k-1}v \ .
\ee
Observe now that
\begin{multline} \label{BB4}
\sum_{y\in \Z^d} \int_{-r}^\infty ds \ e^{-m^2(s+r)/2}G(z,-r;y,s,\tau_{-z,r}\phi) d H_k(y-x,s+t;\tau_{x-z,r-t}\phi) \\ = 
\sum_{y\in \Z^d} \int_{0}^\infty ds \ e^{-m^2s/2}G(0,0;y,s,\phi) d H_k(y-(x-z),s+(t-r);\tau_{x-z,r-t}\phi) \ ,
\end{multline}
so the expression is just a function of $(x-z,t-r,\phi)$.
We show by induction that for $k\ge 2$,
\be \label{BC4}
\na_\xi u_{k}(\xi,\eta,x,t,\phi) \ = \ 
\sum_{y\in \Z^d} \int_{0}^\infty ds \ e^{-m^2s/2}G(0,0;y,s,\phi) d H_k(y-x,s+t;\tau_{x,-t}\phi) \ .
\ee
Hence the identity (\ref{AY4}) follows from (\ref{BB4}). 

To prove (\ref{BC4}) we first show that it holds for $k=2$. To see this we note from (\ref{J3})  that  $dH_2$
is given by the formula 
\be \label{BD4}
dH_2(y,s;\phi) \ = \  \La\left\{\nabla\nabla^* G(y,-\La s)\right\}^*\exp[\eta s-iy.\xi]  D\tilde{{\bf b}}(\phi(y,s))v \ .
\ee
Now (\ref{BC4}) for $k=2$ follows from (\ref{AT4}), (\ref{AV4}), (\ref{BD4}).  To do the induction step we use the identity
\begin{multline} \label{BE4}
dH_{k+1}(y,s;\phi) \ = \ \La\left\{\nabla\nabla^* G(y,-\La s)\right\}^*\exp[\eta s-iy.\xi]  D\tilde{{\bf b}}(\phi(y,s))\pa_\xi F_k(\xi,\eta,\tau_{y,s}\phi) \\
+\La\int_{0}^\infty e^{-\eta r} \ dr\sum_{z\in \Z^d} \left\{\nabla\nabla^* G(z,\La r)\right\}^*\exp[-iz\cdot\xi] \tilde{{\bf b}}(\phi(z,-r)) \ dH_k(y-z,s+r;\tau_{z,-r}\phi) \ .
\end{multline}
From (\ref{BE4} and the induction hypothesis (\ref{BC4}) we conclude that
\begin{multline} \label{BF4}
\sum_{y\in \Z^d} \int_{0}^\infty ds \ e^{-m^2s/2}G(0,0;y,s,\phi) d H_{k+1}(y-x,s+t;\tau_{x,-t}\phi) \ = \\ 
\tilde{T}_{\xi,\eta} [f_{k+1}(\xi,\eta,\cdot,\cdot,\phi)](x,t) \  = \ \na_\xi u_{k+1}(\xi,\eta,x,t,\phi) \ .
\end{multline}
\end{proof}
In order to prove Lemma 3.1 we shall need to use a parabolic version of Meyer's theorem.
\begin{lem}
Let $\mathcal{H}$ be a Hilbert space and for $1<p<\infty$ let $L^p(\Z^d\times\R,\mathcal{H})$ be the space of $p$ integrable functions $f:\Z^d\times\R\ra\mathcal{H}$ where the $L^p$ norm of $f$ is defined by 
\be \label{BJ4}
\|f\|^p_p \ =  \ \sum_{y\in\Z^d}\int_{-\infty}^\infty ds \ \|f(y,s)\|_{\mathcal{H}}^p \ .
\ee
Then there exists a constant $C_1>0$ depending only on $d$ such that if $(\xi,\eta)$ lies in the region (\ref{AD4}) the operator $\tilde{T}_{\xi,\eta}$ of (\ref{AT4}) is bounded on $L^p(\Z^d\times\R,\mathcal{H})$ for $1<p<\infty$. Furthermore there is a constant $C_2>0$ depending only on $d$ such that the norm of $\tilde{T}_{\xi,\eta}$ acting on $L^p(\Z^d\times\R,\mathcal{H})$ satisfies the inequality $\|\tilde{T}_{\xi,\eta}\|_p \ \le \ [1+\del(p)](1+C_2|\Im\xi|^2/[\Re\eta/\La])$ where the function
$\del(\cdot)$ depends only on $d$ and $\lim_{p\ra 2}\del(p)=0$. 
\end{lem} 
\begin{proof}
This follows from the argument of Lemma 5.2 and Corollary 5.1 of \cite{cf1}. 
\end{proof}
\begin{proof}[Proof of Lemma 3.1]
We assume $g:\Z^d\times\R\ra\C^d\otimes\C^d$ is continuous of compact support. Then from Lemma 4.2 and Lemma 4.3 we have that
\begin{multline} \label{BG4}
\left\langle \ \left|T_{k,\Im\xi,\Re\eta} \ g(\Re\xi,\Im\eta,\cdot) v\right|^2 \   \right\rangle_{\Om_m} \ \le \ 2\sum_{z\in\Z^d}\int_{-\infty}^\infty dr  \\
 \left\langle  \ \left|\sum_{x\in\Z^{d}}\int_{-\infty}^\infty  dt \  g(x,t) \ e^{m^2(t-r)/2}G(0,0;x-z,r-t,\phi)D\tilde{\mathbf{b}}(\phi(x-z,r-t))\pa_\xi F_k(\xi,\eta,\tau_{x-z,r-t}\phi)\right|^2 \  \right\rangle_{\Om_m}  \\ +
 \left\langle  \ \left|\sum_{x\in \Z^d} \int_{-\infty}^\infty dt \  g(x,t)\tilde{\mathbf{b}}(\phi(x-z,r-t))\mathcal{P}\na_\xi u_{k}(\xi,\eta,x-z,t-r,\phi)\right|^2  \ \right\rangle_{\Om_m} \ .
\end{multline}
Let $(\Om,\mathcal{F},P)$ be the massless field stationary process corresponding to the limit as $m\ra 0$ of the massive field stationary processes defined by (\ref{O1}), (\ref{P1}), and denote expectation with respect to this measure by $\langle\cdot\rangle_\Om$.  We have from \cite{fs} (see also \cite{c1}) that if $d\ge 3$ then for any $N\ge 1$, continuous bounded function $f:\R^{N}\ra\C$, and $(x_1,t_1),...,(x_N,t_N)\in \Z^d\times\R$,
\be \label{BH4}
\lim_{m\ra 0} \left\langle \ f\big( \ \phi(x_1,t_1),....,\phi(x_N,t_N) \ \big) \ \right\rangle_{\Om_{m}} \ = \ \left\langle \ f\big( \ \phi(x_1,t_1),....,\phi(x_N,t_N) \ \big) \ \right\rangle_{\Om} \ .
\ee
Arguing as in Lemma 4.2  we conclude that if $d\ge 3$ then 
\be \label{BI4}
\lim_{m\ra 0} \left\langle \ \left|T_{k,\Im\xi,\Re\eta} \ g(\Re\xi,\Im\eta,\cdot) v\right|^2 \   \right\rangle_{\Om_m}  \ = \ \left\langle \ \left|T_{k,\Im\xi,\Re\eta} \ g(\Re\xi,\Im\eta,\cdot) v\right|^2 \   \right\rangle_{\Om}  \ .
\ee
Hence to prove Lemma 3.1 it will be sufficient to obtain an upper bound on the RHS of (\ref{BG4}) which is independent of $m$ as $m\ra 0$. 

For  $\xi\in\R^d,\eta\in\C$ with $\Re\eta>0$ the functions $F_k(\xi,\eta,\phi), \ k=2,3,..., \phi\in\Om,$ defined by (\ref{AW4})  satisfy the recurrence equations
\begin{multline} \label{BK4}
[\eta+\pa]F_2(\xi,\eta,\phi)+\La\pa _\xi^*\pa_\xi  F_2(\xi,\eta,\phi) \ = \ \La \mathcal{P}\pa^*_\xi[\tilde{\mathbf{b}}(\phi(0,0))v] \ , \\
[\eta+\pa]F_k(\xi,\eta,\phi)+\La\pa _\xi^*\pa_\xi  F_k(\xi,\eta,\phi) \ = \ \La\mathcal{P}\pa^*_\xi[\tilde{\mathbf{b}}(\phi(0,0))\pa_\xi F_{k-1}(\xi,\eta,\phi)]   \ {\rm if   \ } k>2.
\end{multline}
 Then as in (\ref{B3}) we see that  $F_k(\xi,\eta,\cdot)\in L^2(\Om)$ and
\be \label{BL4}
\frac{\Re\eta}{\La}\|F_k(\xi,\eta,\cdot)\|^2_{L^2(\Om)}+\|\pa_\xi F_k(\xi,\eta,\cdot)\|^2_{\mathcal{H}(\Om)} \ \le \  (1-\la/\La)^{2(k-1)}|v|^2 \ .
\ee
The $F_k(\xi,\eta,\phi)$  for $\xi\in\C^d$ are defined by analytic continuation from the values of $F_k(\xi,\eta,\phi)$ when $\xi\in\R^d$.  From Lemma 2.1 of \cite{cf1} we see that $F_k$ regarded as a function $(\xi,\eta)\ra L^2(\Om)$ has an analytic continuation to the region (\ref{AD4})  where $C_1$ is a constant depending only on $d$. For $(\xi,\eta)$ in this region there is a constant $C_2$ depending only on $d$ such that 
\be \label{BM4}
\frac{\Re\eta}{2\La}\|F_k(\xi,\eta,\cdot)\|^2_{L^2(\Om)}+\|\pa_\xi F_k(\xi,\eta,\cdot)\|^2_{\mathcal{H}(\Om)} \ \le \  (1-\la/\La)^{2(k-1)}\big[1+C_2|\Im\xi|^2\big/(\Re\eta/\La)\ \big]^{2(k-1)}|v|^2 \ .
\ee

We take $\mathcal{H}$ to be the Hilbert space $\mathcal{H}=L^2(\Om,\C^d)$ and for $k=2,3,..,$ let $h_k:\Z^d\times\R\ra\mathcal{H}$ be the function 
\be \label{BN4}
h_k(z,r,\phi) \ = \  D\tilde{{\bf b}}(\phi(z,-r))\pa_\xi F_{k}(\xi,\eta,\tau_{z,-r}\phi) \   e^{m^2r/2}G(0,0;z,-r,\phi) \ .
\ee
Then the first term on the RHS of (\ref{BG4}) is the square of the norm in $L^2(\Z^d\times\R,\mathcal{H})$ of the convolution $g*h_{k-1}$. It follows from (\ref{U1}) and (\ref{BM4}) that if $q>1+2/d$ then
\be \label{BO4}
\|h_k\|_{L^q(\Z^d\times\R,\mathcal{H})} \ \le \ C_{q}\La^{-1/q}\| D\tilde{{\bf b}}(\cdot)\|_\infty(1-\la/\La)^{(k-1)}\big[1+C_2|\Im\xi|^2\big/(\Re\eta/\La)\ \big]^{(k-1)}|v| \ \  ,
\ee
where the constant $C_q$ depends only on $d,q,\la/\La$.  Taking $q<2$ we have by Young's inequality for convolutions that if for $p=2q/(3q-2)$ the function $g$ is in  $L^p(\Z^d\times\R,\C^d\otimes\C^d)$ with  norm (\ref{L3}) then $g*h_k$ is in 
$L^2(\Z^d\times\R,\mathcal{H})$  and
\be \label{BP4}
\|g*h_k\|_{L^2(\Z^d\times\R,\mathcal{H})} \ \le \ C\|g\|_{L^p(\Z^d\times\R,\C^d\otimes\C^d)}  \ \|h_k\|_{L^q(\Z^d\times\R,\mathcal{H})}  \ ,
\ee
for a constant $C$ depending only on $d$.

To bound the second term on the RHS of (\ref{BG4}) we show that the function $(z,r)\ra \na_\xi u_{k}(\xi,\eta,z,r,\phi)$ is in $L^q(\Z^d\times\R,\mathcal{H})$. In fact from (\ref{U1}), (\ref{AV4}) we have that 
\be \label{BQ4}
\|f_{2}(\xi,\eta,\cdot,\cdot,\cdot)\|_{L^q(\Z^d\times\R,\mathcal{H})} \ \le \ 
C_{q}\La^{-1/q}\| D\tilde{{\bf b}}(\cdot)\|_\infty |v|  \ ,
\ee
for a constant $C_q$ depending only on $q,d$. 
Hence  Lemma 4.4 implies that
\be \label{BR4}
\|\na_\xi u_{2}(\xi,\eta,\cdot,\cdot,\cdot)\|_{L^q(\Z^d\times\R,\mathcal{H})} \ \le \ 
C_{q}\La^{-1/q}\| D\tilde{{\bf b}}(\cdot)\|_\infty [1+\del(q)]\big[1+C_2|\Im\xi|^2\big/(\Re\eta/\La)\ \big]|v|  \ .
\ee
  From (\ref{AX4}) we similarly have that for $k\ge 3$,
\begin{multline} \label{BS4}
\|f_{k}(\xi,\eta,\cdot,\cdot,\cdot)\|_{L^q(\Z^d\times\R,\mathcal{H})} \ \le \ 
C_{q}\La^{-1/q}\| D\tilde{{\bf b}}(\cdot)\|_\infty (1-\la/\La)^{(k-2)}\big[1+C_2|\Im\xi|^2\big/(\Re\eta/\La)\ \big]^{(k-2)}|v|   \\
+ (1-\la/\La) \|\na_\xi u_{k-1}(\xi,\eta,\cdot,\cdot,\cdot)\|_{L^q(\Z^d\times\R,\mathcal{H})} \ ,
\end{multline} 
where $C_q$ depends only on $q,d$. Hence using Lemma 4.4 we have by induction  from (\ref{BQ4})-(\ref{BS4}) that for $k\ge 2$, 
\begin{multline} \label{BT4}
\|f_{k}(\xi,\eta,\cdot,\cdot,\cdot)\|_{L^q(\Z^d\times\R,\mathcal{H})} \ \le \\ 
C_{q}k[1+\del(q)]^{k-2}\La^{-1/q}\| D\tilde{{\bf b}}(\cdot)\|_\infty (1-\la/\La)^{(k-2)}\big[1+C_2|\Im\xi|^2\big/(\Re\eta/\La)\ \big]^{(k-2)}|v|  \  .
\end{multline}
The inequality (\ref{N3}) follows now from (\ref{BT4}) and Lemma 4.4 by taking $q$ sufficiently close to $2$ so that  $1+\del(q)<(1-\la/\La)^{-1/2}$. 
\end{proof}

\vspace{.1in}

\section{Proof of Theorem 1.3}
We use an identity relating correlation functions for the massive Euclidean field theory measure (\ref{P1}) and expectations of Green's functions for parabolic PDE with random coefficients. Let $(\Om,\mathcal{F},P_m)$ be the environment of massive fields $\phi:\Z^d\times\R\ra\R$ as in Lemma 4.2 and  $\mathbf{a}(\cdot)$ in (\ref{A1}) be the function $\mathbf{a}(\phi) \ = \ V''(\na\phi(0,0)), \ \phi\in\Om$.  Then the two point correlation function for the invariant measure (\ref{P1}) is related to the expectation of the Green's function for the PDE (\ref{A1}) by
\be \label{A5}
\langle \ \phi(x)\phi(0)  \ \rangle \ = \  \int_0^\infty e^{-m^2 t} G_{\mathbf{a}}(x,t) \ dt \ .
\ee
The identity (\ref{A5}) is implicit in the work of Naddaf and Spencer \cite{ns1} but was first rigorously proven in \cite{gos} (see also \cite{c1}). 

To prove Theorem 1.3 we observe that the methods used to prove Theorem 1.2 for diffusion matrices $\mathbf{a}(\cdot)$ of the form $\mathbf{a}(\phi)=\tilde{\mathbf{a}}(\phi(0,0))$ can also be applied for diffusion matrices of the form   $\mathbf{a}(\phi)=\tilde{\mathbf{a}}(\na\phi(0,0))$, where
$\tilde{\mathbf{a}}:\R^d\ra\R^{d/(d+1)/2}$ is $C^1$, satisfies (\ref{D1}), and $\|D\tilde{\mathbf{a}}(\cdot)\|_\infty<\infty$. Instead of the bound (\ref{U1}) we use the inequality
\be \label{B5}
 |\na_x G(z,s;x,t,\phi)| \ \le \ \frac{C}{[\La(t-s)+1]^{(d+\beta)/2}} \exp\left[-\frac{ |x-z|}{\sqrt{\La (t-s)+1}} \ \right] \ .
\ee
Evidently (\ref{B5}) with $\beta=0$ is a consequence of (\ref{U1}).  This is sufficient to prove Theorem 1.3 when $d\ge 3$. To prove the theorem for $d=2$ we need to use the fact that $\beta$ can be chosen strictly positive depending only on $\la/\La$. This follows from the Harnack inequality \cite {dd}. Furthermore $\beta$ can be chosen arbitrarily close to $1$ provided   $\la/\La$ is sufficiently close to $1$.   The result follows by letting $m\ra 0$ in (\ref{A5}).

\appendix

\section{Diffusion Processes with Convex Potential}
Let $W:\R^k\ra\R$ be a $C^2$ uniformly convex function such that $W''(\cdot)$ satisfies the quadratic form inequality $\la I_k\le W''(\cdot)\le\La I_k$ for some constants $\la,\La>0$. We consider the diffusion process $\phi:\R^+\ra \R^k$ which is the solution to the SDE initial value problem
\be \label{A6}
d\phi(t) \ = \ -\frac{1}{2}\na W(\phi(t)) \ dt+ dB(t) \ , \quad t>0, \quad \phi(0)=0,
\ee
where $B(\cdot)$ is $k$ dimensional Brownian motion. The invariant measure for the SDE (\ref{A6}) is given by
\be \label{B6}
\exp[- W(\phi)] \ d\phi/{\rm normalization} \ , \quad \phi\in\R^k \ .
\ee
We denote the probability space for the stationary process of functions $\phi:\R\ra\R^k$ associated with the SDE (\ref{A6}) and invariant measure (\ref{B6}) by $(\Om,\mathcal{F},P)$, and  expectation  with respect to $(\Om,\mathcal{F},P)$ by $\langle\cdot\rangle_\Om$. 
For $T>0$ let $\phi_T:[-T,\infty)\ra\R^k$ be defined by $\phi_T(t)=\phi(T+t)$, where $\phi(\cdot)$ is the solution to (\ref{A6}). The stationary process measure can be obtained by taking the $T\ra\infty$ limit of $\phi_T$ as follows:
\begin{lem}
Let $f:\R^{Nk}\ra\R$ be a continuous bounded function. Then 
\be \label{C6}
\lim_{T\ra\infty} \langle \ f\big(\phi_T(t_1),..,\phi_T(t_N) \big)\ \rangle \ = \ 
 \langle \ f\big(\phi(t_1),..,\phi(t_N) \big)\ \rangle_\Om \ .
\ee
\end{lem}
\begin{proof}
The diffusion equation corresponding to the SDE (\ref{A6}) is given by
\be \label{D6}
\frac{\pa u(\phi,t)}{\pa t} \ = \  -\frac{1}{2}\na W(\phi)\cdot\na u(\phi,t)+ \frac{1}{2}\Del u(\phi,t) \ , \quad t>0.
\ee 
The solution to (\ref{D6}) with initial data
\be \label{E6}
u(\phi,0) \ = \ u_0(\phi), \quad \phi\in\R^k,
\ee
can be written in terms of the Green's function $G:\R^k\times\R^k\times\R^+\ra\R$ as
\be \label{F6}
u(\phi,t) \ = \ \int_{\R^k} G(\phi,\phi',t) u_0(\phi') \ d\phi' \ , \quad \phi\in\R^k,t>0.
\ee
Now it is clear that for a continuous bounded function $f:\R^k\ra\R$, 
\be \label{G6}
\langle \ f(\phi_T(t_1)) \ \rangle \ =  \ \int_{\R^k}G(0,\phi',T+t_1)f(\phi') \ d\phi' \ .
\ee
Let $\langle\cdot\rangle_W$ denote expectation with respect to the invariant measure (\ref{B6}) and $L^2_W(\R^k)$ the corresponding space of square integrable functions $g:\R^k\ra\R$ with respect to  $\langle\cdot\rangle_W$. Letting $[\cdot,\cdot]_W$ denote the inner product on $L^2_W(\R^k)$ we see from (\ref{G6}) that for any $\del$ satisfying $0<\del<T+t_1$, 
\begin{multline} \label{H6}
\langle \ f(\phi_T(t_1)) \ \rangle \ =  \ [f_1,f_2]_W \quad {\rm  where \ } f_1(\phi)=G(0,\phi,\del) \exp[W(\phi)] \int_{\R^k} \exp[- W(\phi')] \ d\phi' \ \\
{\rm and \ }  f_2(\phi) \ = \  \int_{\R^k}G(\phi,\phi',T+t_1-\del)f(\phi') \ d\phi' \ .
\end{multline}
Since $\|f_2\|_\infty\le \|f\|_\infty$ it follows that $f_2\in L^2_W(\R^k)$.  We can also easily see that for $\del>0$ sufficiently small the function $f_1$ is in $L^2_W(\R^k)$.  Now we use the fact that the operator $H=-\Del+\na W(\phi)\cdot\na$ is self adjoint non-negative definitive on $L^2_W(\R^k)$ and the constant is an eigenfunction of $H$ with eigenvalue $0$.  From the Brascamp-Lieb inequality \cite{bl} the operator $H$ acting on the subspace of $L^2_W(\R^k)$ of functions orthogonal to the constant is bounded below by $\la>0$.  Furthermore from (\ref{H6}) we have that
\be \label{I6}
\langle \ f(\phi_T(t_1)) \ \rangle \ =  \ [f_1,e^{-H(T+t_1-\del)/2}f]_W  \ ,
\ee
whence we conclude that
\be \label{J6}
\lim_{T\ra\infty} \langle \ f(\phi_T(t_1)) \ \rangle \ =  \  \langle \ f_1(\phi) \ \rangle_W \langle \ f(\phi) \ \rangle_W \ = \ \langle \ f(\phi) \ \rangle_W \ .
\ee
We have proven (\ref{C6}) when $N=1$. 
The identity (\ref{C6})  for $N>1$ can be proven similarly. Assuming $0<t_1< t_2<\cdots< t_N$,  we have that
\begin{multline} \label{K6}
\langle \ f\big(\phi_T(t_1),..,\phi_T(t_N) \big)\ \rangle \ = \ [f_1,e^{-H(T+t_1-\del)/2}g]_W \quad {\rm where}  \\
g(\phi_1) \ =  \ \int_{\R^{(N-1)k}} G(\phi_1,\phi_2,t_2-t_1)\cdots G(\phi_{N-1},\phi_N,t_N-t_{N-1}) f(\phi_1,...,\phi_N) \ d\phi_2\cdots d\phi_N  \ .
\end{multline}
Letting $T\ra\infty$ in (\ref{K6}) we see as before that(\ref{C6}) holds. 
\end{proof}
Next we wish to obtain a representation of the measure for the probability space $(\Om,\mathcal{F},P)$ for the stationary process associated with the SDE (\ref{A6}) and invariant measure (\ref{B6}).  First we consider the Gaussian case  so there is a symmetric $k\times k$ matrix and $k$ dimensional vector $b$ with
\be \label{L6}
W(\phi) \ = \ \frac{1}{2}\phi^*A\phi -b^*\phi  \ ,\quad {\rm where \ }   \la I_k\le A\le \La I_k \ .
\ee
The SDE (\ref{A6}) is explicitly solvable when $W(\cdot)$ is given by (\ref{L6}) with solution
\be \label{M6}
\phi(t) \ = \ \int_0^t e^{-A(t-s)/2}  [b/2  \ ds+dB(s)] \ , \quad t\ge 0.
\ee 
It is well known that the measure for the stationary process is Gaussian . We can use Lemma A1 and (\ref{M6}) to find formulas for the mean and covariance of $\phi(\cdot)$. Thus we have that
\be \label{N6}
\langle \ \phi(t) \ \rangle_\Om \ = \  A^{-1} b, \quad {\rm cov}_\Om[\phi(t_1),\phi(t_2)^*] \ = \ \Gamma(t_2-t_1) \ = \   A^{-1}e^{-A|t_1-t_2|/2} \ .
\ee
The Fourier transform (\ref{K1}) of the covariance is therefore given by $\hat{\Gamma}(\theta)=[\theta^2+A^2/4]^{-1}$.  Hence the Gaussian measure corresponding to the covariance is  formally given by the expression
\be \label{O6}
\exp\left[ \ -\frac{1}{2}\int_{-\infty}^\infty \left|\frac{d\phi(t)}{dt}\right|^2+\frac{1}{4}|A\phi(t)-b|^2 \ dt \ \right] \prod_{t\in\R} d\phi(t)\ \bigg/ {\rm normalization} \ .
\ee
Evidently the measure (\ref{O6}) is log concave.

We can obtain a representation of the stationary process measure similar to (\ref{O6}) for general $C^2$ uniformly convex functions $W(\cdot)$. To see this we write the solution of (\ref{D6}), (\ref{E6}) using the Cameron-Martin formula \cite{ks} as
\be \label{P6}
u(\phi,t) \ = \ E\left[ \exp\left\{  - \frac{1}{2}\int_0^t \na W(B(s))\cdot dB(s)-  \frac{1}{8}\int_0^t |\na W(B(s))|^2 \ ds \ \right\} u_0(B(t))  \ \Big| \ B(0)=\phi  \right] \ ,
\ee
where $B(\cdot)$ is $k$ dimensional Brownian motion. We rewrite (\ref{P6}) using Ito's formula
\be \label{Q6}
W(B(t))-W(B(0)) \ = \ \int_0^t \na W(B(s))\cdot dB(s) +\frac{1}{2} \int_0^t \Del W(B(s))\ ds  \ . 
\ee
From  (\ref{P6}), (\ref{Q6}) we see that
\be \label{R6}
u(\phi,t) \ = \ e^{W(\phi)/2}E\left[ \exp\left\{   -  \frac{1}{2}\int_0^t -\frac{1}{2}\Del W(B(s))+\frac{1}{4} |\na W(B(s))|^2 \ ds \ \right\} e^{-W(B(t))/2}u_0(B(t))  \ \Big| \ B(0)=\phi  \right] \ .
\ee
The identity (\ref{R6}) can be alternatively obtained using  the Feynman-Kac representation \cite{ks}  for the solution to the PDE 
\be \label{S6}
\frac{\pa v(\phi,t)}{\pa t} \ = \  V(\phi) v(\phi,t)+ \frac{1}{2}\Del v(\phi,t) \ , \quad t>0,
\ee 
 with initial data
\be \label{T6}
v(\phi,0) \ = \ v_0(\phi), \quad \phi\in\R^k.
\ee
Thus we have from the Feynman-Kac formula  that
\be \label{U6}
v(\phi,t) \ = \ E\left[ \exp\left\{  \int_0^t V(B(s)) \ ds \ \right\} v_0(B(t))  \ \Big| \ B(0)=\phi  \right] \ .
\ee
The formula (\ref{R6}) follows now from (\ref{U6}) using the fact that if $u(\phi,t)$ is the solution to (\ref{D6}), (\ref{E6}) then the function $v(\phi,t)=\exp[-W(\phi)/2]u(\phi,t)$ is the solution to (\ref{S6}), (\ref{T6}) with
\be \label{V6}
V(\phi) \ = \ \frac{1}{4}\Del W(\phi) -\frac{1}{8}|\na W(\phi)|^2 \ , \quad v_0(\phi) \ = \ \exp[-W(\phi)/2]u_0(\phi) \ .
\ee

To obtain the representation for the measure of the stationary process we use Lemma A1. Thus from (\ref{R6}) the LHS of (\ref{C6}) is given by the formula
\begin{multline} \label{W6}
 \langle \ f\big(\phi_T(t_1),..,\phi_T(t_N) \big)\ \rangle \ = \\
 \ e^{W(0)/2}E\Bigg[ \exp\left\{   -  \frac{1}{2}\int_0^{T+t_N+T'} -\frac{1}{2}\Del W(B(s))+\frac{1}{4} |\na W(B(s))|^2 \ ds \ \right\} \\
  e^{-W(B(T+t_N+T'))/2}f(B(T+t_1), B(T+t_2),..,B(T+t_N))  \ \Big| \ B(0)=0  \Bigg] \ ,
\end{multline}
for any $T'\ge 0$.  Recall now that the Brownian motion measure $B(s), \ s\ge0,$ has the representation
\be \label{X6}
\exp\left[ \ -\frac{1}{2}\int_{0}^\infty \left|\frac{d\phi(s)}{ds}\right|^2 ds \ \right] \prod_{s\in\R^+} d\phi(s)\ \bigg/ {\rm normalization} \  \quad {\rm with \  } \phi(0)=0 \ .
\ee
Hence on letting $T,T'\ra\infty$ in (\ref{W6}) and using (\ref{X6}) we see that limiting measure defined by  (\ref{C6}) with expectation $\langle\cdot\rangle_\Om$ has the representation 
\be \label{Y6}
\exp\left[ \ -\frac{1}{2}\int_{-\infty}^\infty \left|\frac{d\phi(t)}{dt}\right|^2-\frac{1}{2}\Del W(\phi(t))+\frac{1}{4} |\na W(\phi(t))|^2 \ \ dt \ \right] \prod_{t\in\R} d\phi(t)\ \bigg/ {\rm normalization} \ .
\ee
In the Gaussian case (\ref{L6}) the representation (\ref{Y6}) is equivalent to (\ref{O6}) since $\Del W(\cdot)$ is constant and is therefore  part of the normalization constant. The measure (\ref{Y6}) is log concave when $W(\cdot)$ is quadratic, but it is easy to see that even if $W(\cdot)$ is a  small perturbation of a quadratic the measure is no longer log concave.

 \thanks{ {\bf Acknowledgement:}  The authors would like to thank Tom Hurd and Tom Spencer for helpful conversations.  }


\begin{thebibliography}{9}
\bibitem{aronson}
D.~Aronson,
\newblock {\em Bounds for the fundamental solution of a parabolic equation},
\newblock Bull. Am. Math. Soc. {\bf 73} (1967), 890-896.
\bibitem{bal}
G.~Bal, J.~Garnier, S.~Motsch and V.~Perrier,
\newblock{\em Random integrals and correctors in homogenization},
\newblock Asymptotic Analysis {\bf 59} (2008), 1-26. 
\bibitem{bl}
H.~Brascamp and E.~Lieb,
\newblock {\em On extensions of the Brunn-Minkowski and Pr\'ekopa-Leindler
theorems, including inequalities for log concave functions, and with an
application to the diffusion equation},
\newblock J. Functional Analysis {\bf 22} (1976), 366-389.
\bibitem{bmp}
C.~Boldrighini, R.~Minlos and A.~Pellegrinotti,
\newblock {\em Random walks in quenched i.i.d. space-time environments are always a.s. diffusive},
\newblock  Probab. Theory Related Fields {\bf 129} (2004), 133-156.
\bibitem{brydges}
D.~ Brydges,
\newblock{\em Lectures on the renormalisation group},
\newblock  Statistical Mechanics, pp. 7-93, IAS/Park City Math. Ser., {\bf 16}, Amer. Math. Soc., Providence 2009. 
\bibitem{cafs}
 L.~Caffarelli and P.~Souganidis,
\newblock {\em Rates of convergence for the homogenization of fully nonlinear uniformly elliptic pde in random media,}
\newblock Invent. Math. {\bf 180} (2010), 301-360.
\bibitem{ct}
R.~Carmona and M.~Tehranchi,
\newblock {\em Interest Rate Models: an Infinite Dimensional Stochastic Analysis Perspective,}
\newblock Springer-Verlag, Berlin-Heidelberg 2006.
\bibitem{cps}
G.~Checkin, A.~Piatnitski and A.~Shamaev,
\newblock{\em Homogenization: Methods and Applications,}
\newblock Translations of Mathematical Monographs {\bf 234}, Amer. Math. Soc., Providence, 2007.
\bibitem{c1}
J.~Conlon,
\newblock{ \em PDE with Random Coefficients and Euclidean Field Theories},
\newblock J. Statistical Physics {\bf 114} (2004), 933-958.
\bibitem{c2}
J.~Conlon,
\newblock {\em Greens Functions for Elliptic and Parabolic Equations with
Random Coefficients II},  
\newblock Transactions of AMS {\bf 356} (2004), 4085-4142. 
\bibitem{cf1}
J.~Conlon and A.~Fahim,
\newblock {\em Strong Convergence to the homogenized limit of parabolic equations with random coefficients,} Transactions of AMS, to appear.
\newblock 
\bibitem{cf2}
J.~Conlon and A.~Fahim,
\newblock {\em Strong Convergence to the homogenized limit of elliptic equations with random coefficients II,} 
\newblock Bulletin of the London Math. Soc., to appear.
\bibitem{cs}
J.~Conlon and T.~Spencer,
\newblock {\em Strong Convergence to the homogenized limit of elliptic equations with random coefficients,}
\newblock Transactions of AMS, to appear.
\bibitem{dd}
T.~Delmotte and J.~Deuschel,
\newblock {\em On estimating the derivatives of symmetric diffusions in
stationary random environment,  with applications to $\na\phi$ interface model},
\newblock Probab. Theory Relat. Fields {\bf 133}, 358-390 (2005). 
\bibitem{dh}
J.~Dimock and T.~Hurd,
\newblock{\em A Renormalization Group Analysis of Correlation Functions for the Dipole Gas},
\newblock J. Statistical Physics {\bf 66} (1992), 1277-1318.
\bibitem{dkl} D.~Dolgopyat, G.~Keller and C.~ Liverani,
\newblock{\em Random walk in Markovian Environment},
\newblock Ann. Probab. {\bf 36} (2008),  1676-1710.
\bibitem{fs} 
T.~Funaki and H.~Spohn,
\newblock{\em Motion by mean curvature from the Ginzburg-Landau 
$\nabla \phi$ interface model},
\newblock Comm. Math. Phys. {\bf 185} (1997), 1-36.
\bibitem{gos}
G.~Giacomin, S.~Olla and H.~Spohn,
\newblock {\em Equilibrium fluctuations for $\nabla\phi$ interface model},
\newblock Ann. Probab. {\bf 29} (2001), 1138-1172.
\bibitem{go1}
A.~ Gloria and F.~Otto,
\newblock {\em An optimal variance estimate in stochastic homogenization of discrete elliptic equations},
\newblock  Ann. Probab. {\bf 39} (2011), 779-856.
\bibitem{go2}
A.~ Gloria and F.~Otto,
\newblock {\em An optimal error estimate in stochastic homogenization of discrete elliptic equations},
\newblock  Ann. Appl. Probab. {\bf 22} (2012), 1-28.
\bibitem{gno}
A.~Gloria, S.~Neukamm and F.~Otto,
\newblock {\em Quantification of ergodicity in stochastic homogenization: optimal bounds via spectral gap on Glauber dynamics,}
\newblock 2013 preprint.
\bibitem{hs}
B.~Helffer and J.~Sj\"ostrand,
\newblock {\em On the correlation for Kac-like models in the convex case},
\newblock J. Statist. Phys. {\bf 74} (1994), 349-409.
\bibitem{j}
B. ~Frank Jones,
\newblock {\em A class of singular integrals},
\newblock Amer. J. Math. {\bf 86} (1964),  441-462.
\bibitem{ks}
I.~Karatzas and S.~Shreve,
\newblock {\em Brownian Motion and Stochastic Calculus,}
\newblock Second Edition, Graduate Texts in Mathematics {\bf 113}, Springer-Verlag, New York 1991. 
\bibitem{k1} 
S.~Kozlov,
\newblock {\em Averaging of random structures},
\newblock Dokl. Akad. Nauk. SSSR {\bf 241} (1978), 1016-1019.
\bibitem{k2} 
S.~Kozlov,
\newblock {\em The method of averaging and walks in inhomogeneous environment},
\newblock Russian Math. Surveys {\bf 40} (1985), 73-145.
\bibitem{loy}
C.~Landim, S. Olla and H.~T.~Yau,
\newblock{\em Convection-diffusion equation with space-time ergodic random flow,}
\newblock Probab. Theory Relat. Fields {\bf 112} (1998), 203-220.
\bibitem{mo}
D.~Marahrens and F.~Otto,
\newblock {\em Annealed Estimates on the Green's function,}
\newblock 2013 preprint.
\bibitem{m}
N.~Meyers,
\newblock {\em An $L^p$ estimate for the gradient of solutions of second order elliptic divergence equations},
\newblock Ann. Scuola Norm. Pisa Cl. Sci. {\bf 17} (1963),  189-206.
\bibitem{mourrat}
J.-C. Mourrat,
\newblock{\em Kantorivich distance in the martingale CLT and quantitative homogenization of parabolic equations with random coefficients},
\newblock 2012 preprint- arXiv: 1203.3417.
\bibitem{ns1}
A.~Naddaf and T.~Spencer,
\newblock {\em On homogenization and scaling limit of some gradient
perturbations of a massless free field},
\newblock Comm. Math. Phys. {\bf 183} (1997), 55-84.
\bibitem{ns2}
A.~Naddaf and T.~Spencer,
\newblock {\em Estimates on the variance of some homogenization problems},
\newblock 1998 preprint.
\bibitem{nu}
D.~Nualart,
\newblock{\em The Malliavin Calculus and related topics, second edition,}
\newblock Springer Verlag, 2005.
\bibitem{pv} 
G.~Papanicolaou  and S.~Varadhan,  
\newblock {\em Boundary value problems with rapidly oscillating random 
coefficients}, 
\newblock Volume 2 of \textit{
Coll. Math. Soc. Janos Bolya}, \textbf{27},  Random fields, Amsterdam,
North Holland Publ. Co. 1981, pp. 835-873.
\bibitem{r}
R.~Rhodes,
\newblock{\em On homogenization of space-time dependent and degenerate random flows,}
\newblock Stochastic Processes and their Applications {\bf 117} (2007), 1561-1585.
\bibitem{stein}
E.~Stein,
\newblock{\em Singular Integrals and Differentiability Properties of Functions},
\newblock Princeton University Press, Princeton, N.J. 1970.
\bibitem{y}
V.~Yurinskii,
\newblock{\em Averaging of symmetric diffusion in random medium},
\newblock Sibirskii Matematicheskii Zhurnal {\bf 27} (1986), 167-180. 
\bibitem{zko}
V.~Zhikov, S.~Kozlov and O.~Oleinik,
\newblock{\em Homogenization of Differential Operators and Integral
Functionals},
\newblock Springer Verlag, Berlin, 1994.
\end{thebibliography}
\end{document}